 \documentclass[final,5p,times,twocolumn,english,fleqn]{elsarticle}
\pdfoutput=1
 \usepackage{graphics}
 \usepackage{graphicx}
\usepackage{epsfig}
\usepackage{epstopdf}

\usepackage{amssymb}
\usepackage{amsthm}
\usepackage{amsmath}
\usepackage{multirow}
\usepackage{array}
\usepackage{float}
\usepackage{threeparttable}
\usepackage{algorithm}
\usepackage{algpseudocode}
\usepackage{babel}

\biboptions{compress}

\journal{Computer Methods in Applied Mechanics and Engineering}

\begin{document}
\begin{frontmatter}

\title{Adaptive-Sparse Polynomial Dimensional Decomposition Methods for
High-Dimensional Stochastic Computing}

\author{Vaibhav Yadav}
\ead{vaibhav-yadav@uiowa.edu}

\author{Sharif Rahman\corref{cor1}\fnref{fn1}}
\ead{rahman@engineering.uiowa.edu}

\cortext[cor1]{Corresponding author.}
\fntext[fn1]{Professor.}
\tnotetext[t1]{Grant sponsor: U.S. National Science Foundation; Grant Nos. CMMI-0653279, CMMI-1130147.}

\address{College of Engineering, The University of Iowa, Iowa City, Iowa 52242, U.S.A.}
\begin{abstract}
This article presents two novel adaptive-sparse polynomial dimensional
decomposition (PDD) methods for solving high-dimensional uncertainty
quantification problems in computational science and engineering.
The methods entail global sensitivity analysis for retaining important PDD component functions,
and a full- or sparse-grid dimension-reduction integration or quasi
Monte Carlo simulation for estimating the PDD expansion coefficients.
A unified algorithm, endowed with two distinct ranking schemes for grading component functions, was created for their numerical implementation. The fully adaptive-sparse PDD method is comprehensive and rigorous,
leading to the second-moment statistics of a stochastic response that
converges to the exact solution when the tolerances vanish. A partially
adaptive-sparse PDD method, obtained through regulated adaptivity
and sparsity, is economical and is, therefore, expected to solve practical
problems with numerous variables. Compared with past developments,
the adaptive-sparse PDD methods do not require its truncation parameter(s)
to be assigned \emph{a priori} or arbitrarily. The numerical results
reveal that an adaptive-sparse PDD method achieves a desired level
of accuracy with considerably fewer coefficients compared with existing
PDD approximations. For a required accuracy in calculating the probabilistic response characteristics, the new bivariate adaptive-sparse PDD method is more
efficient than the existing bivariately truncated PDD method by almost
an order of magnitude. Finally, stochastic dynamic analysis of a disk
brake system was performed, demonstrating the ability of the new methods
to tackle practical engineering problems.
\end{abstract}


\end{frontmatter}

\newtheorem{thm}{Theorem}
\newtheorem{pro}{Proposition}
\newdefinition{rmk}{Remark}

\section{Introduction}

\label{intro} Uncertainty quantification, an emerging multidisciplinary
field blending physical and mathematical sciences, characterizes the
discrepancy between model-based simulations and physical reality in
terms of the statistical moments, probability law, and other relevant
properties of a complex system response. For practical applications,
encountering a large number of input random variables is not uncommon,
where an output function of interest, defined algorithmically via
expensive finite-element analysis (FEA) or similar numerical calculations,
is all too often expensive to evaluate. The most promising stochastic
methods available today are perhaps the collocation \cite{deb01,zabaras07} and
polynomial chaos expansion (PCE) \cite{ghanem91,xiu02} methods, including sparse-grid techniques \cite{holtz08}, which have found many successful applications. However,
for truly high-dimensional systems, they require astronomically large
numbers of terms or coefficients, succumbing to the curse of dimensionality
\cite{bellman57}. Therefore, alternative computational methods capable of exploiting
low effective dimensions of multivariate functions, such as the polynomial
dimensional decomposition (PDD) method, are most desirable. Readers,
not familiar with but interested in PDD, are referred to the authors'
past works \cite{rahman08,rahman09,rahman10,rahman11}.

For practical applications, the PDD must be truncated with respect
to $S$ and $m$, where $S$ and $m$ define the largest degree of interactions among
input variables and largest order of orthogonal polynomials, respectively, retained
in a concomitant approximation. These truncation parameters depend
on the dimensional structure and nonlinearity of a stochastic response.
The higher the values of $S$ and $m$, the higher the accuracy, but
also the computational cost that is endowed with an $S$th- or $m$th-order
polynomial computational complexity. However, the dimensional hierarchy
or nonlinearity, in general, is not known \emph{a priori}. Therefore,
indiscriminately assigning the truncation parameters is not desirable,
nor is it possible to do so when a stochastic solution is obtained
via complex numerical algorithms. In which case, one must perform
these truncations automatically by progressively drawing in higher-variate
or higher-order contributions as appropriate. Furthermore, all $S$-variate
component functions of PDD may not contribute equally or even appreciably
to be considered in the resulting approximation. Hence, a sparse approximation,
expelling component functions with negligible contributions, should
be considered as well.

Addressing some of the aforementioned concerns have
led to adaptive versions of the cut-high-dimensional model representation
(cut-HDMR) \cite{zabaras10} and the anchored decomposition \cite{karniadakis12}, employed
in conjunction with the sparse-grid collocation methods, for solving
stochastic problems in fluid dynamics. Several adaptive variants of the PCE \cite{blatman08, li98, wan05} method have also appeared. It is important to clarify
that the cut-HDMR and anchored decompositions are the same as the
referential dimensional decomposition (RDD) \cite{rahman11c, rahman11b}. Therefore, both adaptive methods essentially
employ RDD for multivariate function approximations, where the mean
values of random input are treated as the reference or anchor point
$-$ a premise originally proposed by Xu and Rahman \cite{xu05}. The
developments of these adaptive methods were motivated by the fact
that an RDD approximation requires only function evaluations, as opposed
to high-dimensional integrals required for an ANOVA Dimensional Decomposition (ADD) approximation. However,
a recent error analysis \cite{rahman11b} reveals sub-optimality of RDD approximations,
meaning that an RDD approximation, regardless of how the reference
point is chosen, cannot be better than an ADD approximation for identical
degrees of interaction. The analysis also finds ADD approximations
to be exceedingly more precise than RDD approximations at higher-variate
truncations. In addition, the criteria implemented in existing adaptive
methods are predicated on retaining higher-variate component functions
by examining the second-moment properties of only univariate component
functions, where the largest degree of interaction and polynomial
order in the approximation are still left to the user's discretion,
instead of being determined automatically based on the problem being
solved. Therefore, more intelligently derived adaptive-sparse approximations
and decompositions rooted in ADD or PDD should be explored by developing
relevant criteria and acceptable error thresholds. These enhancements,
some of which are indispensable, should be pursued without sustaining
significant additional cost.

This paper presents two new adaptive-sparse versions of the PDD method -- the fully adaptive-sparse PDD method and a partially adaptive-sparse PDD method --
for solving high-dimensional stochastic problems commonly encountered
in computational science and engineering. The methods are based on
(1) variance-based global sensitivity analysis for defining the pruning
criteria to retain important PDD component functions; (2) a full-
or sparse-grid dimension-reduction integration or quasi Monte Carlo
simulation (MCS) for estimating the PDD expansion coefficients. Section
2 briefly describes existing dimensional decompositions, including
PDD and its $S$-variate, $m$th-order approximation, to be contrasted
with the proposed methods. Two adaptive-sparse PDD methods are formally
presented in Section 3, along with a computational algorithm and a
flowchart for numerical implementation of the method. Two different
approaches for calculating the PDD coefficients, one emanating
from dimension-reduction integration and the other employing quasi
MCS, are explained in Section 4. Section 5 presents
three numerical examples for probing the accuracy, efficiency, and
convergence properties of the proposed methods, including a comparison
with the existing PDD methods. Section 6 reports a large-scale stochastic
dynamics problem solved using a proposed adaptive-sparse method. Finally,
conclusions are drawn in Section 7.

\section{Dimensional Decompositions}

Let $\mathbb{N}$, $\mathbb{N}_{0}$, $\mathbb{R}$, and $\mathbb{R}_{0}^{+}$
represent the sets of positive integer (natural), non-negative integer,
real, and non-negative real numbers, respectively. For $k\in\mathbb{N}$,
denote by $\mathbb{R}^{k}$ the $k$-dimensional Euclidean space,
by $\mathbb{N}_{0}^{k}$ the $k$-dimensional multi-index space, and
by $\mathbb{R}^{k\times k}$ the set of $k\times k$ real-valued matrices.
These standard notations will be used throughout the paper.

Let $(\Omega,\mathcal{F},P)$ be a complete probability space, where
$\Omega$ is a sample space, $\mathcal{F}$ is a $\sigma$-field on
$\Omega$, and $P:\mathcal{F}\to[0,1]$ is a probability measure.
With $\mathcal{B}^{N}$ representing the Borel $\sigma$-field on
$\mathbb{R}^{N}$, $N\in\mathbb{N}$, consider an $\mathbb{R}^{N}$-valued
random vector $\mathbf{X}:=(X_{1},\cdots,X_{N}):(\Omega,\mathcal{F})\to(\mathbb{R}^{N},\mathcal{B}^{N})$,
which describes the statistical uncertainties in all system and input
parameters of a high-dimensional stochastic problem. The probability
law of $\mathbf{X}$ is completely defined by its joint probability
density function $f_{\mathbf{X}}:\mathbb{R}^{N}\to\mathbb{R}_{0}^{+}$.
Assuming independent coordinates of $\mathbf{X}$, its joint probability
density $f_{\mathbf{X}}(\mathbf{x})=\Pi_{i=1}^{i=N}f_{i}(x_{i})$
is expressed by a product of marginal probability density functions
$f_{i}$ of $X_{i}$, $i=1,\cdots,N$, defined on the probability
triple $(\Omega_{i},\mathcal{F}_{i},P_{i})$ with a bounded or an
unbounded support on $\mathbb{R}$. For a given $u\subseteq\{1,\cdots,N\}$,
$f_{\mathbf{X}_{-u}}(\mathbf{x}_{-u}):=\prod_{i=1,i\notin u}^{N}f_{i}(x_{i})$
defines the marginal density function of $\mathbf{X}_{-u}:=\mathbf{X}_{\{1,\cdots,N\}\backslash u}$.

\subsection{ANOVA Dimensional Decomposition}

Let $y(\mathbf{X}):=y(X_{1},\cdots,X_{N}$), a real-valued, measurable
transformation on $(\Omega,\mathcal{F})$, define a stochastic response to a high-dimensional random input and $\mathcal{L}_{2}(\Omega,\mathcal{F},P)$
represent a Hilbert space of square-integrable functions $y$ with
respect to the induced generic measure $f_{\mathbf{X}}(\mathbf{x})d\mathbf{x}$
supported on $\mathbb{R}^{N}$. The ANOVA dimensional decomposition,
expressed by the recursive form \cite{efron81,rahman11b,sobol03}
\begin{align}
y(\mathbf{X}) & ={\displaystyle \sum_{u\subseteq\{1,\cdots,N\}}y_{u}(\mathbf{X}_{u})},\label{1a}\\
y_{\emptyset} & =\int_{\mathbb{R}^{N}}y(\mathbf{x})f_{\mathbf{X}}(\mathbf{x})d\mathbf{x},\label{1b}\\
y_{u}(\mathbf{X}_{u}) & ={\displaystyle \int_{\mathbb{R}^{N-|u|}}y(\mathbf{X}_{u},\mathbf{x}_{-u})}f_{\mathbf{X}_{-u}}(\mathbf{x}_{-u})d\mathbf{\mathbf{x}}_{-u}-{\displaystyle \sum_{v\subset u}}y_{v}(\mathbf{X}_{v}),\label{1c}
\end{align}
is a finite, hierarchical expansion in terms of its input variables
with increasing dimensions, where $u\subseteq\{1,\cdots,N\}$ is a
subset with the complementary set $-u=\{1,\cdots,N\}\backslash u$
and cardinality $0\le|u|\le N$, and $y_{u}$ is a $|u|$-variate
component function describing a constant or the interactive effect
of $\mathbf{X}_{u}=(X_{i_{1}},\cdots,X_{i_{|u|}})$, $1\leq i_{1}<\cdots<i_{|u|}\leq N$,
a subvector of $\mathbf{X}$, on $y$ when $|u|=0$ or $|u|>0$. The
summation in Equation \eqref{1a} comprises $2^{N}$ terms, with each
term depending on a group of variables indexed by a particular subset
of $\{1,\cdots,N\}$, including the empty set $\emptyset$.

The ADD component functions $y_{u}$, $u\subseteq\{1,\cdots,N\}$,
have two remarkable properties: (1) the component functions,
$y_{u}$, $\emptyset\ne u\subseteq\{1,\cdots,N\}$, have \emph{zero} means;
and (2) two distinct component functions $y_{u}$ and $y_{v}$, where
$u\subseteq\{1,\cdots,N\}$, $v\subseteq\{1,\cdots,N\}$, and $u\neq v$,
are orthogonal \cite{rahman11b}. However, the ADD component functions are
difficult to obtain, because they require calculation of high-dimensional
integrals.

\subsection{Referential Dimensional Decomposition}

Consider a reference point $\mathbf{c}=(c_{1},\cdots,c_{N})\in\mathbb{R}^{N}$
and the associated Dirac measure $\prod_{i=1}^{N}\delta(x_{i}-c_{i})dx_{i}$.
The referential dimensional decomposition is created when $\prod_{i=1}^{N}\delta(x_{i}-c_{i})dx_{i}$
replaces the probability measure in Equations \eqref{1a}-\eqref{1c},
leading to the recursive form
\begin{align}
y(\mathbf{X}) & ={\displaystyle \sum_{u\subseteq\{1,\cdots,N\}}w_{u}(\mathbf{X}_{u};\mathbf{c})},\label{5a}\\
w_{\emptyset} & =y(\mathbf{c}),\label{5b}\\
w_{u}(\mathbf{X}_{u};\mathbf{c}) & =y(\mathbf{X}_{u},\mathbf{c}_{-u})-{\displaystyle \sum_{v\subset u}}w_{v}(\mathbf{X}_{v};\mathbf{c}),\label{5c}
\end{align}
also known as cut-HDMR \cite{rabitz99}, anchored decomposition \cite{hickernell96},
and anchored-ANOVA decomposition \cite{griebel10}, with the latter two referring
to the reference point as the anchor. Xu and Rahman introduced Equations
\eqref{5a}-\eqref{5c} with the aid of Taylor series expansion, calling
them dimension-reduction \cite{xu04} and decomposition \cite{xu05} methods
for statistical moment and reliability analyses, respectively, of
mechanical systems. Compared with ADD, RDD lacks orthogonal features,
but its component functions are easier to obtain as they only involve
function evaluations at a chosen reference point.

\subsection{Polynomial Dimensional Decomposition}

Let $\{\psi_{ij}(X_{i});\; j=0,1,\cdots\}$ be a set of orthonormal
polynomial basis functions in the Hilbert space $\mathcal{L}_{2}(\Omega_{i},\mathcal{F}_{i},P_{i})$
that is consistent with the probability measure $P_{i}$ of $X_{i}$, where $i = 1,\cdots,N$.
For a given $\emptyset\ne u=\{i_{1},\cdots,i_{|u|}\}\subseteq\{1,\cdots,N\}$,
$1\le|u|\le N$, $1\le i_{1}<\cdots<i_{|u|}\le N$, denote a product
probability triple by $(\times_{p=1}^{p=|u|}\Omega_{i_{p}},\times_{p=1}^{p=|u|}\mathcal{F}_{i_{p}},\times_{p=1}^{p=|u|}P_{i_{p}})$,
and the associated space of square integrable $|u|$-variate component
functions of $y$ by $\mathcal{L}_{2}(\times_{p=1}^{p=|u|}\Omega_{i_{p}},\times_{p=1}^{p=|u|}\mathcal{F}_{i_{p}},\times_{p=1}^{p=|u|}P_{i_{p}}):=\{y_{u}:\int_{\mathbb{R}^{|u|}}y_{u}^{2}(\mathbf{x}_{u})f_{\mathbf{X}_{u}}(\mathbf{x}_{u})d\mathbf{x}_{u}<\infty\}$,
which is a Hilbert space. Since the joint density of $(X_{i_{1}},\cdots,X_{i_{|u|}})$
is separable (independence), \emph{i.e.}, $f_{\mathbf{X}_{u}}(\mathbf{x}_{u})={\textstyle \prod_{p=1}^{|u|}}f_{i_{p}}(x_{i_{p}})$,
the product polynomial $\psi_{u\mathbf{j}_{|u|}}(\mathbf{X}_{u}):=\prod_{p=1}^{|u|}\psi_{i_{p}j_{p}}(X_{i_{p}})$,
where $\mathbf{j}_{|u|}=(j_{1},\cdots,j_{|u|})\in\mathbb{N}_{0}^{|u|}$,
a $|u|$-dimensional multi-index with $\infty$-norm $\left\Vert \mathbf{j}_{|u|}\right\Vert _{\infty}:=\max(j_{1},\cdots,j_{|u|})$,
constitutes an orthonormal basis in $\mathcal{L}_{2}(\times_{p=1}^{p=|u|}\Omega_{i_{p}},$
$\times_{p=1}^{p=|u|}\mathcal{F}_{i_{p}},$ $\times_{p=1}^{p=|u|}P_{i_{p}})$.

The orthogonal polynomial expansion of a non-constant
$|u|$-variate component function becomes \cite{rahman08,rahman09}
\begin{equation}
y_{u}(\mathbf{X}_{u})=\sum_{{\textstyle {\mathbf{j}_{|u|}\in\mathbb{N}_{0}^{|u|}\atop j_{1},\cdots,j_{|u|}\neq0}}}C_{u\mathbf{j}_{|u|}}\psi_{u\mathbf{j}_{|u|}}(\mathbf{X}_{u}),\:\emptyset\ne u\subseteq\{1,\cdots,N\},\label{6}
\end{equation}
 with
\begin{equation}
C_{u\mathbf{j}_{|u|}}:=\int_{\mathbb{R}^{N}}y(\mathbf{x})\psi_{u\mathbf{j}_{|u|}}(\mathbf{\mathbf{x}}_{u})f_{\mathbf{X}}(\mathbf{x})d\mathbf{x},\;\emptyset\ne u\subseteq\{1,\cdots,N\},\;\mathbf{j}_{|u|}\in\mathbb{N}_{0}^{|u|},\label{7}
\end{equation}
representing the corresponding expansion coefficient. The end result
of combining Equations \eqref{1a} and \eqref{6} is the PDD \cite{rahman08,rahman09},
\begin{equation}
y(\mathbf{X})=y_{\emptyset}+{\displaystyle \sum_{\emptyset\ne u\subseteq\{1,\cdots,N\}}}\:\sum_{{\textstyle {\mathbf{j}_{|u|}\in\mathbb{N}_{0}^{|u|}\atop j_{1},\cdots,j_{|u|}\neq0}}}C_{u\mathbf{j}_{|u|}}\psi_{u\mathbf{j}_{|u|}}(\mathbf{X}_{u}),\label{8}
\end{equation}
providing an exact, hierarchical expansion of $y$ in terms of an
infinite number of coefficients or orthonormal polynomials. All component functions
$y_{u}$, $\emptyset\ne u\subseteq\{1,\cdots,N\}$, in Equation \eqref{6} have \emph{zero} means and
satisfy the orthogonal properties of the ADD. Therefore,
PDD can be viewed as the polynomial version of ADD,
inheriting all desirable properties of ADD.

\subsection{Truncated Dimensional Decompositions}

The three dimensional decompositions $-$ ADD, RDD and PDD $-$ are grounded on a fundamental conjecture
known to be true in many real-world applications: given a high-dimensional
function $y$, its $|u|$-variate component functions decay rapidly
with respect to $|u|$, leading to accurate lower-variate approximations
of $y$. Indeed, given the integers $0\le S<N$ and $1\le m<\infty$ for all $1\le|u|\le S$,
the truncated dimensional decompositions
\begin{align}
\tilde{y}_{S}(\mathbf{X}) & ={\displaystyle \sum_{{\textstyle {u\subseteq\{1,\cdots,N\}\atop 0\le|u|\le S}}}y_{u}(\mathbf{X}_{u})},\label{sr1}\\
\hat{y}_{S}(\mathbf{X};\mathbf{c}) & =\sum_{{\textstyle {u\subseteq\{1,\cdots,N\}\atop 0\le|u|\le S}}}w_{u}(\mathbf{X}_{u};\mathbf{c}),\label{sr2}\\
\tilde{y}_{S,m}(\mathbf{X}) & =y_{\emptyset}+{\displaystyle \sum_{{\textstyle {\emptyset\ne u\subseteq\{1,\cdots,N\}\atop 1\le|u|\le S}}}}\:\sum_{{\textstyle {\mathbf{j}_{|u|}\in\mathbb{N}_{0}^{|u|},\left\Vert \mathbf{j}_{|u|}\right\Vert _{\infty}\le m\atop j_{1},\cdots,j_{|u|}\neq0}}}C_{u\mathbf{j}_{|u|}}\psi_{u\mathbf{j}_{|u|}}(\mathbf{X}_{u}),\label{sr3}
\end{align}
respectively, describe $S$-variate ADD, RDD, and PDD approximations,
which for $S>0$ include interactive effects of at most $S$ input
variables $X_{i_{1}},\cdots,X_{i_{S}}$, $1\leq i_{1}<\cdots<i_{S}\leq N$,
on $y$. It is elementary to show that when $S\to N$ and/or $m\to\infty$,
$\tilde{y}_{S}$, $\hat{y}_{S}$, and $\tilde{y}_{S,m}$ converge
to $y$ in the mean-square sense, generating a hierarchical and convergent
sequence of approximation of $y$ from each decomposition.

\subsubsection{ADD and RDD Errors}

For ADD or RDD to be useful, what are the approximation errors committed
by $\tilde{y}_{S}(\mathbf{X})$ and $\hat{y}_{S}(\mathbf{X};\mathbf{c})$
in Equations \eqref{sr1} and \eqref{sr2}? More importantly, for
a given $0\le S<N$ , which approximation between ADD and RDD is better?
Since the RDD approximation depends on the reference point $\mathbf{c}$, no analytical error analysis is possible if $\mathbf{c}$ is deterministic or arbitrarily chosen. However, if $\mathbf{c}$ follows the same probability measure of $\mathbf{X}$, then the error committed by an RDD approximation on average can be compared with the error from an ADD approximation, as follows.

\begin{thm}
\label{thm1}
Let $\mathbf{c}=(c_{1},\cdots,c_{N})\in\mathbb{R}^{N}$
be a random vector with the joint probability density function of
the form $f_{\mathbf{X}}(\mathbf{c})=\prod_{j=1}^{j=N}f_{j}(c_{j})$,
where $f_{j}$ is the marginal probability density function of the
$j$th coordinate of $\mathbf{X}=(X_{1},\cdots,X_{N})$. Define two
second-moment errors
\begin{equation}
e_{S,A}:=\mathbb{E}\left[\left(y(\mathbf{X})-\tilde{y}_{S}(\mathbf{X})\right)^{2}\right]:=\int_{\mathbb{R}^{N}}\left[y(\mathbf{x})-\tilde{y}_{S}(\mathbf{x})\right]^{2}f_{\mathbf{X}}(\mathbf{x})d\mathbf{x},\label{sr4}
\end{equation}
and
\begin{equation}
\begin{split}
e_{S,R}(\mathbf{c})&:=\mathbb{E}\left[\left(y(\mathbf{X})-\hat{y}_{S}(\mathbf{X};\mathbf{c})\right)^{2}\right]\\
&:=\int_{\mathbb{R}^{N}}\left[y(\mathbf{x})-\hat{y}_{S}(\mathbf{x};\mathbf{c})\right]^{2}f_{\mathbf{X}}(\mathbf{x})d\mathbf{x},\label{sr5}
\end{split}
\end{equation}
committed by the $S$-variate ADD and RDD approximations, respectively,
of $y(\mathbf{X})$. Then the lower and upper bounds of the expected
error $\mathbb{E}\left[e_{S,R}\right]:=\int_{\mathbb{R}^{N}}e_{S,R}(\mathbf{c})f_{\mathbf{X}}(\mathbf{c})d\mathbf{c}$
from the $S$-variate RDD approximation, expressed in terms of the
error $e_{S,A}$ from the $S$-variate ADD approximation, are
\begin{equation}
2^{S+1}e_{S,A}\le\mathbb{E}\left[e_{S,R}\right]\le\left[{\displaystyle 1+\sum_{k=0}^{S}}\binom{N-S+k-1}{k}^{2}\binom{N}{S-k}\right]e_{S,A},\label{sr6}
\end{equation}
where $0\le S<N,$ $S+1\le N<\infty$.
\end{thm}

\begin{proof}
See Theorem 4.12 and Corollary 4.13 of Rahman \cite{rahman11b}.
\end{proof}

\begin{rmk}
Theorem \ref{thm1} reveals that the expected error
from the univariate ($S=1$) RDD approximation is at least four times
larger than the error from the univariate ADD approximation. In contrast,
the expected error from the bivariate ($S=2$) RDD approximation can
be eight or more times larger than the error from the bivariate ADD
approximation. Given an arbitrary truncation, an ADD approximation
is superior to an RDD approximation. In addition, the RDD approximations
may perpetrate very large errors at upper bounds when there exist
a large number of variables and appropriate conditions. Therefore,
existing adaptive methods \cite{zabaras10,karniadakis12} anchored in RDD approximations
should be used with caveat. Furthermore, the authors advocate using
PDD for adaptivity, but doing so engenders its own computational challenges,
to be explained in the forthcoming sections.
\end{rmk}

\subsubsection{Statistical Moments of PDD}

Applying the expectation operator on $\tilde{y}_{S,m}(\mathbf{X})$
and $(\tilde{y}_{S,m}(\mathbf{X})-y_{\emptyset})^{2}$ and noting the
\emph{zero}-mean and orthogonal properties of PDD component functions, the mean \cite{rahman10}
\begin{equation}
\mathbb{E}\left[\tilde{y}_{S,m}(\mathbf{X})\right]=y_{\emptyset}\label{10}
\end{equation}
of the $S$-variate, $m$th-order PDD approximation matches the exact
mean $\mathbb{E}\left[y(\mathbf{X})\right]$, regardless of $S$ or
$m$, and the approximate variance \cite{rahman10}
\begin{equation}
\begin{split}
\sigma_{S,m}^{2}&:=\mathbb{E}\left[\left(\tilde{y}_{S,m}(\mathbf{X})-\mathbb{E}\left[\tilde{y}_{S,m}(\mathbf{X})\right]\right)^{2}\right]\\
&={\displaystyle \sum_{{\textstyle {\emptyset\ne u\subseteq\{1,\cdots,N\}\atop 1\le|u|\le S}}}}\:\sum_{{\textstyle {\mathbf{j}_{|u|}\in\mathbb{N}_{0}^{|u|},\left\Vert \mathbf{j}_{|u|}\right\Vert _{\infty}\le m\atop j_{1},\cdots,j_{|u|}\neq0}}}C_{u\mathbf{j}_{|u|}}^{2}\label{11}
\end{split}
\end{equation}
is calculated as the sum of squares of the expansion coefficients
from the $S$-variate, $m$th-order PDD approximation of $y(\mathbf{X})$.
Clearly, the approximate variance approaches the exact variance \cite{rahman10}
\begin{equation}
\sigma^{2}:=\mathbb{E}\left[\left(y(\mathbf{X})-\mathbb{E}\left[y(\mathbf{X})\right]\right)^{2}\right]={\displaystyle {\displaystyle \sum_{\emptyset\ne u\subseteq\{1,\cdots,N\}}}}\:\sum_{{\textstyle {\mathbf{j}_{|u|}\in\mathbb{N}_{0}^{|u|}\atop j_{1},\cdots,j_{|u|}\neq0}}}C_{u\mathbf{j}_{|u|}}^{2}\label{sr7}
\end{equation}
of $y$ when $S\to N$ and $m\to\infty$. The mean-square convergence
of $\tilde{y}_{S,m}$ is guaranteed as $y$, and its component functions
are all members of the associated Hilbert spaces.

The $S$-variate, $m$th-order PDD approximation $\tilde{y}_{S,m}(\mathbf{X})$
in Equation \eqref{sr3} contains
\begin{equation}
\tilde{K}_{S,m}=1+{\displaystyle \sum_{{\textstyle {\emptyset\ne u\subseteq\{1,\cdots,N\}\atop 1\le|u|\le S}}}}\:{\displaystyle \sum_{{\textstyle {\mathbf{j}_{|u|}\in\mathbb{N}_{0}^{|u|},\left\Vert \mathbf{j}_{|u|}\right\Vert _{\infty}\le m\atop j_{1},\cdots,j_{|u|}\neq0}}}}1=\sum_{k=0}^{S}\binom{N}{k}m^{k}\label{eq:K-tilda-Sm}
\end{equation}
number of PDD coefficients and corresponding orthonormal polynomials.
Therefore, the computational complexity of a truncated PDD is polynomial,
as opposed to exponential, thereby alleviating the curse of dimensionality
to some extent.

\begin{rmk}
Constructing a PDD approximation by pre-selecting
$S$ and/or $m$, unless they are quite small, is computationally
intensive, if not impossible, for high-dimensional uncertainty quantification.
In other words, the existing PDD is neither scalable nor adaptable,
which is crucial for solving industrial-scale stochastic problems.
A requisite theoretical basis and innovative numerical algorithms
for overcoming these limitations are presented in Section 3.
\end{rmk}

\begin{rmk}
The PDD approximation and its second-moment analysis
require the expansion coefficients $C_{u\mathbf{j}_{|u|}}$, which,
according to their definition in Equation \eqref{7}, involve
various $N$-dimensional integrals over $\mathbb{R}^{N}$. For large
$N$, a full numerical integration employing an $N$-dimensional tensor
product of a univariate quadrature formula is computationally prohibitive.
This is one drawback of ADD and PDD, since their component functions
entail calculating high-dimensional integrals. Therefore, novel dimension-reduction
integration schemes or sampling techniques, to be described in Section
4, are needed to estimate the coefficients efficiently.
\end{rmk}

\subsubsection{PDD versus PCE Approximations}

The long form of an $S$-variate, $m$th-order PDD approximation of
$y(\mathbf{X})$ is the expansion \cite{rahman08, rahman09}
\begin{equation}
\begin{array}{ccl}
\tilde{y}_{S,m}(\mathbf{X}) \!\!\!\!\!& := &\!\!\!\!\! y_{0}+{\displaystyle \sum_{i=1}^{N}\sum_{j=1}^{m}C_{ij}\psi_{ij}(X_{i})}+\\
 &  & \!\!\!\!\!{\displaystyle \sum_{i_{1}=1}^{N-1}\:\sum_{i_{2}=i_{1}+1}^{N}}{\displaystyle \;\sum_{j_{2}=1}^{m}\sum_{j_{1}=1}^{m}C_{i_{1}i_{2}j_{1}j_{2}}\psi_{i_{1}j_{1}}(X_{i_{1}})\psi_{i_{2}j_{2}}(X_{i_{2}})}+\cdots+\\
 &  & \!\!\!\!\!\!\!\!{\displaystyle \begin{array}[t]{c}
\underbrace{{\displaystyle \sum_{i_{1}=1}^{N-s+1}\!\! \cdots \!\! \sum_{i_{S}=i_{S-1}+1}^{N}}}\\
{\scriptstyle S\;\mathrm{sums}}
\end{array}\!\!\!\!{\displaystyle \;\begin{array}[t]{c}
\underbrace{{\displaystyle \sum_{j_{1}=1}^{m}\cdots\sum_{j_{S}=1}^{m}}}\\
{\scriptstyle S\;\mathrm{sums}}
\end{array}\!\!C_{i_{1}\cdots i_{S}j_{1}\cdots j_{S}}\prod_{q=1}^{S}\psi_{i_{q}j_{q}}(X_{i_{q}})}}
\end{array}\label{1}
\end{equation}
in terms of random orthonormal polynomials $\psi_{ij}(X_{i})$, $i=1,\cdots,N$,
$j=1,\cdots,m$, of input variables $X_{1},\cdots,X_{N}$ with increasing
dimensions, where $y_{0}$ and $C_{i_{1}\cdots i_{S}j_{1}\cdots j_{S}}$,
$1\le i_{1}<\cdots<i_{S}\le N$, $j_{1},\cdots,j_{S}=1,\cdots,m$,
are the PDD expansion coefficients. In contrast, a $p$th-order PCE
approximation of $y(\mathbf{X})$, where $0\le p<\infty$, has the
representation \cite{ghanem91}
\begin{equation}
\begin{array}{ccl}
\check{y}_{p}(\mathbf{X}) & := & a_{0}\Gamma_{0}+{\displaystyle \sum_{i=1}^{N}a_{i}\Gamma_{1}(X_{i})}+{\displaystyle \sum_{i_{1}=1}^{N}\sum_{i_{2}=i_{1}}^{N}a_{i_{1}i_{2}}\Gamma_{2}(X_{i_{1}},X_{i_{2}})}\\
 &  & +\cdots+{\displaystyle \sum_{i_{1}=1}^{N}\cdots\sum_{i_{p}=i_{p-1}}^{N}}a_{i_{1}\cdots i_{p}}\Gamma_{p}(X_{i_{1}},\cdots,X_{i_{p}})
\end{array}\label{2}
\end{equation}
in terms of random polynomial chaoses $\Gamma_{p}(X_{i_{1}},\cdots,X_{i_{p}})$,
$1\le i_{1}\le\cdots\le i_{p}\le N$, of input variables $X_{i_{1}},\cdots,X_{i_{p}}$
with increasing orders, where $a_{0}$ and $a_{i_{1}\cdots i_{p}}$
are the PCE expansion coefficients. The polynomial chaoses are various
combinations of tensor products of sets of univariate orthonormal
polynomials. Therefore, both expansions share the same orthonormal
polynomials, and their coefficients require evaluating similar high-dimensional
integrals.

\begin{rmk}
The PDD and PCE when truncated are not the same.
In fact, two important observations jump out readily. First, the terms
in the PCE approximation are organized with respect to the order of
polynomials. In contrast, the PDD approximation is structured with
respect to the degree of interaction between a finite number of random
variables. Therefore, significant differences may exist regarding
the accuracy, efficiency, and convergence properties of their truncated
sum or series. Second, if a stochastic response is highly nonlinear,
but contains rapidly diminishing interactive effects of multiple random
variables, the PDD approximation is expected to be more effective
than the PCE approximation. This is because the lower-variate (univariate,
bivariate, \emph{etc.}) terms of the PDD approximation can be just
as nonlinear by selecting appropriate values of $m$ in Equation \eqref{1}.
In contrast, many more terms and expansion coefficients are required
to be included in the PCE approximation to capture such high nonlinearity.
\end{rmk}

In reference to a past study \cite{rahman11}, consider two mean-squared errors,
$e_{S,m}:=\mathbb{E}[y(\mathbf{X})-\tilde{y}_{S,m}(\mathbf{X})]^{2}$
and $e_{p}:=\mathbb{E}[y(\mathbf{X})-\check{y}_{p}(\mathbf{X})]^{2}$,
owing to the $S$-variate, $m$th-order PDD approximation $\tilde{y}_{S,m}(\mathbf{X})$
and $p$th-order PCE approximation $\check{y}_{p}(\mathbf{X})$, respectively,
of $y(\mathbf{X})$. For a class of problems where the interactive
effects of $S$ input variables on a stochastic response get progressively
weaker as $S\to N$, then the PDD and PCE errors for identical expansion
orders can be weighed against each other. For this special case, set
$m=p$ and assume that ${\displaystyle C_{i_{1}\cdots i_{s}j_{1}\cdots j_{s}}=0}$,
where $s=S+1,\cdots,N$, $1\le i_{1}<\cdots<i_{s}\le N$, $j_{1},\cdots j_{s}=1,\cdots,\infty$.
Then it can be shown that $e_{m}\ge e_{S,m}$, demonstrating larger
error from the PCE approximation than from the PDD approximation \cite{rahman11}.
In the limit, when $S=N$, $e_{m}\ge e_{N,m}$, regardless of the
values of the expansions coefficients. In other words, the $N$-variate,
$m$th-order PDD approximation cannot be worse than the $m$th-order
PCE approximation. When $S<N$ and ${\displaystyle C_{i_{1}\cdots i_{s}j_{1}\cdots j_{s}}}$,
$s=S+1,\cdots,N$, $1\le i_{1}<\cdots<i_{s}\le N$, $j_{1},\cdots j_{s}=1,\cdots,\infty$,
are not negligible and arbitrary, numerical convergence analysis is
required for comparing these two errors. Indeed, numerical analyses
of mathematical functions or simple dynamic systems reveal markedly
higher convergence rates of the PDD approximation than the PCE approximation \cite{rahman11}.
From the comparison of computational efforts, required to estimate
with the same precision the frequency distributions of complex dynamic
systems, the PDD approximation can be significantly more efficient
than the PCE approximation \cite{rahman11}.

\section{Proposed Adaptive-Sparse PDD Methods}

\subsection{Global Sensitivity Indices}

The global sensitivity analysis quantifies how an output function
of interest is influenced by individual or subsets of input variables,
illuminating the dimensional structure lurking behind a complex response.
Indeed, these sensitivity indices have been used to rank variables,
fix unessential variables, and reduce dimensions of large-scale problems
\cite{rahman11d, sobol01}. The authors propose to exploit these indices, developed
in conjunction with PDD, for adaptive-sparse PDD approximations as
follows.

The global sensitivity index of $y(\mathbf{X})$ for a subset $\mathbf{X}_{u}$,
$\emptyset\ne u\subseteq\{1,\cdots,N\}$, of input variables $\mathbf{X},$
denoted by $G_{u}$, is defined as the non-negative ratio \cite{rahman11d, sobol01}
\begin{equation}
G_{u}:=\frac{\mathbb{E}\left[y_{u}^{2}(\mathbf{X})\right]}{\sigma^{2}},0<\sigma^{2}<\infty,\label{sr8}
\end{equation}
representing the fraction of the variance of $y(\mathbf{X})$ contributed
by the ADD component function $y_{u}$. Since $\emptyset\ne u\subseteq\{1,\cdots,N\}$,
there exist $2^{N}-1$ such indices, adding up to $\sum_{u\subseteq\{1,\cdots,N\}}G_{u}=1$.
Applying the Fourier-polynomial approximation of $y_{u}(\mathbf{X})$,
that is, Equation \eqref{6}, and noting the properties of orthonormal polynomials, the component variance
\begin{equation}
\mathbb{E}\left[y_{u}^{2}(\mathbf{X}_{u})\right]=\sum_{{\textstyle {\mathbf{j}_{|u|}\in\mathbb{N}_{0}^{|u|}\atop j_{1},\cdots,j_{|u|}\neq0}}}C_{u\mathbf{j}_{|u|}}^{2}\label{sr9}
\end{equation}
of $y_{u}$ is the sum of squares of its PDD expansion coefficients.
When the right side of Equation \eqref{sr9} is truncated at $\left\Vert \mathbf{j}_{|u|}\right\Vert _{\infty}=m_{u}$, where $m_{u}\in\mathbb{N}$, and then used to replace the numerator
of Equation \eqref{sr8}, the result is an $m_{u}$th-order approximation
\begin{equation}
\tilde{G}_{u,m_{u}}:=\frac{1}{\sigma^{2}}\sum_{{\textstyle {\mathbf{j}_{|u|}\in\mathbb{N}_{0}^{|u|},\left\Vert \mathbf{j}_{|u|}\right\Vert _{\infty}\le m_{u}\atop j_{1},\cdots,j_{|u|}\neq0}}}C_{u\mathbf{j}_{|u|}}^{2},\label{sr10}
\end{equation}
which approaches $G_{u}$ as $m_{u}\to\infty$. Given $2\le m_{u}<\infty$, consider two approximate global sensitivity
indices $\tilde{G}_{u,m_{u}-1}$ and $\tilde{G}_{u,m_{u}}$for $\mathbf{X}_{u}$
such that $\tilde{G}_{u,m_{u}-1}\ne0$. Then the normalized index,
defined by
\begin{equation}
\Delta\tilde{G}_{u,m_{u}}:=\frac{\tilde{G}_{u,m_{u}}-\tilde{G}_{u,m_{u}-1}}{\tilde{G}_{u,m_{u}-1}},\;\tilde{G}_{u,m_{u}-1}\ne0,\label{sr11}
\end{equation}
represents the relative change in the approximate global sensitivity
index when the largest polynomial order increases from $m_{u}-1$
to $m_{u}$. The sensitivity indices $\tilde{G}_{u,m_{u}}$ and $\Delta\tilde{G}_{u,m_{u}}$
provide an effective means to truncate the PDD in Equation \eqref{8}
both adaptively and sparsely.

\subsection{The Fully Adaptive-Sparse PDD Method}

Let $\epsilon_{1}\ge0$ and $\epsilon_{2}\ge0$ denote two non-negative
error tolerances that specify the minimum values of $\tilde{G}_{u,m_{u}}$
and $\Delta\tilde{G}_{u,m_{u}}$, respectively. Then a fully adaptive-sparse
PDD approximation
\begin{equation}
\bar{y}(\mathbf{X}):=y_{\emptyset}+{\displaystyle \sum_{\emptyset\ne u\subseteq\{1,\cdots,N\}}}\:{\displaystyle \sum_{m_{u}=1}^{\infty}}\:\sum_{{\textstyle {\left\Vert \mathbf{j}_{|u|}\right\Vert _{\infty}=m_{u},\, j_{1},\cdots,j_{|u|}\neq0\atop \tilde{G}_{u,m_{u}}>\epsilon_{1},\Delta\tilde{G}_{u,m_{u}}>\epsilon_{2}}}}C_{u\mathbf{j}_{|u|}}\psi_{u\mathbf{j}_{|u|}}(\mathbf{X}_{u})\label{29}
\end{equation}
of $y(\mathbf{X})$ is formed by the subset of PDD component functions,
satisfying two inclusion criteria: (1) $\tilde{G}_{u,m_{u}}>\epsilon_{1}$,
and (2) $\Delta\tilde{G}_{u,m_{u}}>\epsilon_{2}$ for all $1\le|u|\le N$
and $1\le m_{u}<\infty$. The first criterion requires the contribution
of an $m_{u}$-th order polynomial approximation of $y_{u}(\mathbf{X})$
towards the variance of $y(\mathbf{X})$ to exceed $\epsilon_{1}$
in order to be accommodated within the resultant truncation. The second
criterion identifies the augmentation in the variance contribution
from $y_{u}(\mathbf{X}_{u})$ evoked by a single increment in the
polynomial order $m_{u}$ and determines if it surpasses $\epsilon_{2}$.
In other words, these two criteria ascertain which interactive effects
between two or more input random variables are retained and dictate the largest
order of polynomials in a component function, formulating
a fully adaptive-sparse PDD approximation.

When compared with the PDD in Equation \eqref{8}, the adaptive-sparse
PDD approximation in Equation \eqref{29} filters out the relatively
insignificant component functions with a scant compromise on the accuracy
of the resulting approximation. Furthermore, there is no need to pre-select
the truncation parameters of the existing PDD approximation. The level
of accuracy achieved by the fully adaptive-sparse PDD is meticulously
controlled through the tolerances $\epsilon_{1}$ and $\epsilon_{2}$.
The lower the tolerance values, the higher the accuracy of the approximation.
It is elementary to show that the mean-squared error in the fully adaptive-sparse
PDD approximation disappears when the tolerances vanish, that is, $\bar{y}(\mathbf{X})$ approaches $y(\mathbf{X})$
as $\epsilon_{1}\to0$, $\epsilon_{2}\to0$.

\subsection{A Partially Adaptive-Sparse PDD Method}

Based on the authors' past experience, an $S$-variate PDD approximation,
where $S\ll N$, is adequate, when solving real-world engineering
problems, with the computational cost varying polynomially ($S$-order)
with respect to the number of variables \cite{rahman08,rahman09}. As an example,
consider the selection of $S=2$ for solving a stochastic problem
in 100 dimensions by a bivariate PDD approximation, comprising $100\times99/2=4950$
bivariate component functions. If all such component functions are
included, then the computational effort for even a full bivariate
PDD approximation may exceed the computational budget allocated to
solving this problem. But many of these component functions contribute
little to the probabilistic characteristics sought and can be safely
ignored. Similar conditions may prevail for higher-variate component
functions. Henceforth, define an $S$-variate, partially adaptive-sparse
PDD approximation
\begin{equation}
\bar{y}_{S}(\mathbf{X}):=y_{\emptyset}+{\displaystyle \sum_{{\textstyle {\emptyset\ne u\subseteq\{1,\cdots,N\}\atop 1\le|u|\le S}}}}\:{\displaystyle \sum_{m_{u}=1}^{\infty}}\:\sum_{{\textstyle {\left\Vert \mathbf{j}_{|u|}\right\Vert _{\infty}=m_{u},\, j_{1},\cdots,j_{|u|}\neq0\atop \tilde{G}_{u,m_{u}}>\epsilon_{1},\Delta\tilde{G}_{u,m_{u}}>\epsilon_{2}}}}C_{u\mathbf{j}_{|u|}}\psi_{u\mathbf{j}_{|u|}}(\mathbf{X}_{u})\label{30}
\end{equation}
of $y(\mathbf{X})$, which is attained by subsuming at most $S$-variate
component functions, but fulfilling two relaxed inclusion criteria:
(1) $\tilde{G}_{u,m_{u}}>\epsilon_{1}$ for $1\le|u|\le S\le N$, and
(2) $\Delta\tilde{G}_{u,m_{u}}>\epsilon_{2}$ for $1\le|u|\le S\le N$.
Again, the same two criteria are used for the degree of interaction
and the order of orthogonal polynomial, but the truncations are restricted
to at most $S$-variate component functions of $y$.

An $S$-variate, partially adaptive-sparse PDD approximation behaves
differently from the $S$-variate, $m$th-order PDD approximation.
While the latter approximation includes a sum containing at most $S$-variate
component functions, the former approximation may or may not include
all such component functions, depending on the tolerance $\epsilon_{1}$.
For $\epsilon_{1}>0$, an $S$-variate, partially adaptive-sparse
PDD will again trim the component functions with meager contributions.
However, unlike $\bar{y}$ converging to $y$, $\bar{y}_{S}$ converges
to the $S$-variate ADD approximation $\tilde{y}_{S}$, when $\epsilon_{1}\to0$, $\epsilon_{2}\to0$. If $S=N$, then both partially and fully
adaptive-sparse PDD approximations coincide for identical tolerances.

As $S \to N$, $\tilde{y}_{S}(\mathbf{X}) \to y(\mathbf{X})$ in the mean square sense. Given a rate at which $\sigma_{u}^2 := \mathbb{E}\left[y_{u}^2(\mathbf{X}_u)\right]$, the variance of an $|u|$-variate ADD component function, decreases with $|u|$, what can be inferred on how fast $\tilde{y}_{S}(\mathbf{X})$ converges to $y(\mathbf{X})$? Proposition \ref{pro1} and subsequent discussions provide some insights.

\begin{pro}
\label{pro1}
If the variance of a \emph{zero}-mean ADD component function $y_u$ diminishes according to $\sigma_{u}^2 \leq cq^{-|u|}$, where $\emptyset \neq u \subseteq \left\{1,\cdots,N\right\}$, and $c>0$ and $q>1$ are two real-valued constants, then the mean-squared error committed by $\tilde{y}_{S}(\mathbf{X})$, $0 \leq S \leq N$, is
\begin{equation}
\tilde{e}_{S} := \mathbb{E} \left[y(\mathbf{X})-\tilde{y}_{S}(\mathbf{X})\right]^2 \leq c \displaystyle \sum_{s=S+1}^{N}
{N\choose s}
q^{-s}. \label{pro1err}
\end{equation}

\begin{proof}
The result of Proposition \ref{pro1} follows by substituting the expressions of $y(\mathbf{X})$ and $\tilde{y}_{S}(\mathbf{X})$ from Equations \eqref{1a} and \eqref{sr1}, and then using $\sigma_{u}^2 := \mathbb{E}\left[y_{u}^2(\mathbf{X}_u)\right] \leq cq^{-|u|}$.
\end{proof}
\end{pro}

When the equality holds, $\tilde{e}_S$ decays strictly monotonically with respect to $S$ for any rate parameter $q$. The higher the value of $S$, the faster $\tilde{y}_{S}(\mathbf{X})$ converges to $y(\mathbf{X})$ in the mean-square sense.

\subsection{Stochastic Solutions}

\subsubsection{Second-Moment Properties }

Applying the expectation operator on $\bar{y}(\mathbf{X})$ and $\bar{y}_{S}(\mathbf{X})$
and recognizing the \emph{zero}-mean and orthogonal properties of PDD component functions, the means
\begin{equation}
\mathbb{E}\left[\bar{y}(\mathbf{X})\right]=\mathbb{E}\left[\bar{y}_{S}(\mathbf{X})\right]=y_{\emptyset}\label{sr12}
\end{equation}
of fully and partially adaptive-sparse PDD approximations both also
agree with the exact mean $\mathbb{E}\left[y(\mathbf{X})\right]=y_{\emptyset}$
for any $\epsilon_{1}$, $\epsilon_{2}$, and $S$. However, the respective
variances, obtained by applying the expectation operator on $(\bar{y}(\mathbf{X})-y_{\emptyset})^{2}$
and $(\bar{y}_{S}(\mathbf{X})-y_{\emptyset})^{2}$, vary according to
\begin{equation}
\begin{split}
\bar{\sigma}^{2}&:=\mathbb{E}\left[\left(\bar{y}(\mathbf{X})-\mathbb{E}\left[\bar{y}(\mathbf{X})\right]\right)^{2}\right]\\
&={\displaystyle \sum_{\emptyset\ne u\subseteq\{1,\cdots,N\}}}\:{\displaystyle \sum_{m_{u}=1}^{\infty}}\:\sum_{{\textstyle {\left\Vert \mathbf{j}_{|u|}\right\Vert _{\infty}=m_{u},\, j_{1},\cdots,j_{|u|}\neq0\atop \tilde{G}_{u,m_{u}}>\epsilon_{1},\Delta\tilde{G}_{u,m_{u}}>\epsilon_{2}}}}C_{u\mathbf{j}_{|u|}}^{2}\label{sr13}
\end{split}
\end{equation}
and
\begin{equation}
\begin{split}
\bar{\sigma}_{S}^{2}&:=\mathbb{E}\left[\left(\bar{y}_{S}(\mathbf{X})-\mathbb{E}\left[\bar{y}_{S}(\mathbf{X})\right]\right)^{2}\right]\\
&={\displaystyle \sum_{{\textstyle {\emptyset\ne u\subseteq\{1,\cdots,N\}\atop 1\le|u|\le S}}}}\:{\displaystyle \sum_{m_{u}=1}^{\infty}}\:\sum_{{\textstyle {\left\Vert \mathbf{j}_{|u|}\right\Vert _{\infty}=m_{u},\, j_{1},\cdots,j_{|u|}\neq0\atop \tilde{G}_{u,m_{u}}>\epsilon_{1},\Delta\tilde{G}_{u,m_{u}}>\epsilon_{2}}}}C_{u\mathbf{j}_{|u|}}^{2},\label{sr14}
\end{split}
\end{equation}
where the squares of the expansion coefficients are summed following
the same two pruning criteria discussed in the preceding subsections.
Equations \eqref{sr12}-\eqref{sr14} provide closed-form expressions
of the approximate second-moment properties of any square-integrable
function $y$ in terms of the PDD expansion coefficients.

When $\epsilon_{1}=\epsilon_{2}=0$, the right sides of Equations \eqref{sr13}
and \eqref{sr7} coincide, whereas the right side of Equation \eqref{sr14}
approaches that of Equation \eqref{11} for $m\to\infty$. As a consequence,
the variance from the fully adaptive-sparse PDD approximation $\bar{y}(\mathbf{X})$
converges to the exact variance of $y(\mathbf{X})$ as $\epsilon_{1}\to0$
and $\epsilon_{2}\to0$. In contrast, the variance from the $S$-variate,
partially adaptive-sparse PDD approximation $\bar{y}_{S}(\mathbf{X})$
does not follow suit, as it converges to the variance of the $S$-variate,
$m$th-order PDD approximation $\tilde{y}_{S,m}(\mathbf{X})$ as $\epsilon_{1}\to0$
and $\epsilon_{2}\to0$, provided that $m\to\infty$. Therefore, the
fully adaptive-sparse PDD approximation is more rigorous than a partially
adaptive-sparse PDD approximation, but the latter can be more useful
than the former when solving practical engineering problems and will
be demonstrated in the Numerical Examples and Application sections.

\subsubsection{Probability Distribution}

Although the PDD approximations are mean-square convergent, Equations
\eqref{29} and \eqref{30} can also be used to estimate higher-order
moments and probability distributions, including rare-event probabilities,
of sufficiently smooth stochastic responses. In this paper, the probability
distribution of $y(\mathbf{X})$ was approximated by performing Monte
Carlo simulation of $\bar{y}(\mathbf{X})$ and/or $\bar{y}_{S}(\mathbf{X})$.
This simulation of the PDD approximation should not be confused with
\emph{crude} Monte Carlo simulation. The crude Monte Carlo method,
which commonly requires numerical calculations of $y$ for input samples
can be expensive or even prohibitive, particularly when the sample
size needs to be very large for estimating small failure probabilities.
In contrast, the Monte Carlo simulation embedded in a PDD approximation
requires evaluations of simple analytical functions. Therefore, an
arbitrarily large sample size can be accommodated in the PDD approximation.

It is also possible to estimate the probability distribution of $y(\mathbf{X})$
from the knowledge of the cumulant generating function of a PDD approximation,
provided that it exists, and then exploit the saddle point approximation
for obtaining an exponential family of approximate distributions.
Readers interested in this alternative approach are referred to the
authors' ongoing work on stochastic sensitivity analysis \cite{ren13}.

It is important to emphasize that the two truncation criteria proposed are strictly based on variance as a measure of output uncertainty. They are highly relevant when the second-moment properties of complex response is desired. For higher-order moments or rare-event probabilities, it is possible to develop alternative sensitivity indices and related pruning criteria. They are not considered here.

\subsection{Numerical Implementation}

The application of fully and partially adaptive-sparse PDD approximations
described by Equations \eqref{29} and \eqref{30} requires selecting
PDD component functions $y_{u}(\mathbf{X}_{u})$, $\emptyset\ne u\subseteq\{1,\cdots,N\}$
and assigning largest orders of their orthogonal polynomial expansions
$1\le m_{u}<\infty$ efficiently such that $\tilde{G}_{u,m_{u}}>\epsilon_{1}$
and $\Delta\tilde{G}_{u,m_{u}}>\epsilon_{2}$ . This section presents
a unified computational algorithm and an associated flowchart developed
to accomplish numerical implementation of the two proposed methods.

\subsubsection{A Unified Algorithm}

The iterative process for constructing an adaptive-sparse PDD approximation,
whether full or partial, comprises two main stages: (1) continue incrementing
the polynomial order $m_{u}$ for a chosen component function $y_{u}(\mathbf{X}_{u})$
unless the criterion $\Delta\tilde{G}_{u,m_{u}}>\epsilon_{2}$ fails;
and (2) continue selecting the component functions $y_{u}(\mathbf{X}_{u})$,
$\emptyset\ne u\subseteq\{1,\cdots,N\}$, unless the criterion $\tilde{G}_{u,m_{u}}>\epsilon_{1}$
fails. These two stages are first executed over all univariate PDD
component functions $y_{u}(\mathbf{X}_{u})$, $|u|=1$, before progressing
to all bivariate component functions $y_{u}(\mathbf{X}_{u})$, $|u|=2$,
and so on, until $|u|=N$ for the fully adaptive-sparse PDD approximation
or until $|u|=S$ for a partially adaptive-sparse PDD approximation,
where $S$ is specified by the user. The implementation details of
the iterative process is described in Algorithm \ref{adaptive-algorithm}
and through the flowchart in Figure \ref{flowchart}.

The first stage of the algorithm presented is predicated on accurate
calculations of the sensitivity indices $\tilde{G}_{u,m_{u}}$ and
$\Delta\tilde{G}_{u,m_{u}}$, which require the variance $\sigma^{2}$
of $y(\mathbf{X})$ as noted by Equations \eqref{sr10} and \eqref{sr11}.
Since there exist an infinite number of expansion coefficients emanating
from all PDD component functions, calculating the variance exactly
from Equation \eqref{sr7} is impossible. To overcome this quandary, the authors
propose to estimate the variance by utilizing all PDD expansion coefficients
available at a juncture of the iterative process. For instance, let
$v\in V$ be an element of the index set $V\subseteq\{1,\cdots,N\}$,
which comprises the subsets of $\{1,\cdots,N\}$ selected so far at
a given step of the iterative process. Then the approximate variance
\begin{equation}
\tilde{\sigma}_{V}^{2}={\displaystyle \sum_{\emptyset\ne v\in V\subseteq\{1,\cdots,N\}}}\:\sum_{{\textstyle {\mathbf{j}_{|v|}\in\mathbb{N}_{0}^{|v|},\left\Vert \mathbf{j}_{|v|}\right\Vert _{\infty}\le m_{v}\atop j_{1},\cdots,j_{|v|}\neq0}}}C_{v\mathbf{j}_{|v|}}^{2}\label{sr15}
\end{equation}
replacing the exact variance $\sigma^{2}$ in Equations \eqref{sr10}
and \eqref{sr11} facilitates an effective iterative scheme for estimating
$\tilde{G}_{u,m_{u}}$ and $\Delta\tilde{G}_{u,m_{u}}$ as well. Equation
\eqref{sr15} was implemented in the proposed algorithm, as explained
in Algorithm \ref{adaptive-algorithm} and Figure \ref{flowchart}.
\begin{algorithm*}[!th]
\caption{\label{adaptive-algorithm}Adaptive-sparse polynomial dimensional
decomposition.}
\begin{algorithmic}
\item Define $S$ \Comment [$S \leftarrow N$ for Fully adaptive]
\item Define $\epsilon_1, \epsilon_2, \epsilon_3$
\For {$|u|\gets 1$ to $S$}
\item $|v| \gets |u|,v\subseteq\{1,\cdots,N\}$
\item $m_v \gets 0$
\Repeat \Comment [continue incrementing the polynomial order $m_v$ unless the ranking of component functions $y_{v}(\boldsymbol{x}_{v})$ converges]
	\item \hspace {5 mm} $m_v \gets m_v + 1$ \Comment [start with the polynomial order $m_v = 1$]
	\item \hspace {5 mm} calculate $C_{vj_v},{\mathbf{j}_{|v|}\in\mathbb{N}_{0}^{|v|},\left\Vert \mathbf{j}_{|v|}\right\Vert _{\infty}\le m_{v}}$ \Comment [from Equation \eqref{7}]
	\item \hspace {5 mm} calculate $\tilde{\sigma}_{V}^{2}\gets{\sum_{\emptyset\ne v\in V\subseteq\{1,\cdots,N\}}}\:\sum_{{\textstyle {\mathbf{j}_{|v|}\in\mathbb{N}_{0}^{|v|},\left\Vert \mathbf{j}_{|v|}\right\Vert _{\infty}\le m_{v}}}}C_{v\mathbf{j}_{|v|}}^{2}$ \Comment [from Equation \eqref{sr15}]
	\item \hspace {5 mm} calculate $\tilde{G}_{v,m_v} \gets \left(\sum_{{\textstyle {\mathbf{j}_{|v|}\in\mathbb{N}_{0}^{|v|},\left\Vert \mathbf{j}_{|v|}\right\Vert _{\infty}\le m_{v}}}}C_{v\mathbf{j}_{|v|}}^{2}\right) / \tilde{\sigma}_{V}^{2}$\Comment [from Equation \eqref{sr10}]
	\item \hspace {5 mm} rank  $y_{v}(\boldsymbol{x}_{v})$: $y_{v^{\left(1\right)}}(\boldsymbol{x}_{v^{\left(1\right)}})$ to $y_{v^{\left(N\right)}}(\boldsymbol{x}_{v^{\left(N\right)}})$ \Comment [from Algorithm 2]
	\item \hspace {5 mm} Get $L$ \Comment [from Algorithm 2]
	\item \hspace {5 mm} $N_{m_u} \gets 0$
		\For {$i \gets 1$ to $L$} \Comment [comparing rankings from $m_u$ with those from $\left(m_u -1\right)$ to check for convergence]
		\item \hspace {10 mm} $R_{m_u}\left(i\right) \gets i$
		\If {$R_{m_{u}-1}\left(i\right)$ = $R_{m_u}\left(i\right)$} $N_{m_u} \gets N_{m_u} + 1$
		\EndIf
		\EndFor
\Until {$N_m / L \ge \epsilon_3$} \Comment [ranking converge]
\For {$l_u \gets 1$ to $L$} \Comment [start the adaptivity algorithm with the highest ranking $|u|-$variate component function]
\item \hspace {5 mm} $u \gets u^{\left(l_u\right)}$
\Repeat \Comment [continue incrementing the polynomial order $m_u$ unless the adaptivity condition $\triangle\tilde{G}_{u,m_u} > \epsilon_2$ fails]
	\item \hspace {10 mm} $m_u \gets m_u + 1$
	\item \hspace {10 mm} calculate $C_{uj_u},{\mathbf{j}_{|u|}\in\mathbb{N}_{0}^{|u|},\left\Vert \mathbf{j}_{|u|}\right\Vert _{\infty}\le m_{u}}$ \Comment [from Equation \eqref{7}]
	\item \hspace {10 mm} calculate $\tilde{\sigma}_{V}^{2}\gets{\sum_{\emptyset\ne v\in V\subseteq\{1,\cdots,N\}}}\:\sum_{{\textstyle {\mathbf{j}_{|v|}\in\mathbb{N}_{0}^{|v|},\left\Vert \mathbf{j}_{|v|}\right\Vert _{\infty}\le m_{v}}}}C_{v\mathbf{j}_{|v|}}^{2}$ \Comment [from Equation \eqref{sr15}]
	\item \hspace {10 mm} calculate $\tilde{G}_{u,m_u} \gets \left(\sum_{{\textstyle {\mathbf{j}_{|v|}\in\mathbb{N}_{0}^{|v|},\left\Vert \mathbf{j}_{|v|}\right\Vert _{\infty}\le m_{v}}}}C_{v\mathbf{j}_{|v|}}^{2}\right) / \tilde{\sigma}_{V}^{2}$\Comment [from Equation \eqref{sr10}]
	\item \hspace {10 mm}  calculate $\triangle\tilde{G}_{u,m_u} \gets \left(\tilde{G}_{u,m_u} - \tilde{G}_{u,m_u-1}\right)/\tilde{G}_{u,m_u-1}$ \Comment [from Equation \eqref{sr11}]
\Until {$\triangle\tilde{G}_{u,m_u}\le \epsilon_2$}
\If {$\tilde{G}_{u,m_u}\le \epsilon_1$} exit \EndIf \Comment [exit the adaptivity algorithm]
\EndFor
\EndFor
\item calculate $y_{\emptyset}$\Comment [from Equation \eqref{1b}]
\end{algorithmic}
\end{algorithm*}

The second stage of the algorithm requires an efficient procedure
for selecting appropriate PDD component functions that are retained
in an adaptive-sparse PDD approximation. For a given $1\le|u|\le N,$
let $y_{u}(\mathbf{X}_{u})$, $\emptyset\ne u\subseteq\{1,\cdots,N\}$
denote all $|u|$-variate non-constant PDD component functions of
$y$. It is elementary to count the number of these component functions
to be $L_{|u|}=\binom{N}{|u|}$. Depending on the tolerance criteria
specified, some or none of these component functions may contribute
towards the resultant PDD approximation. Since the component functions
are not necessarily hierarchically arranged, determining their relative
significance to PDD approximation is not straightforward. Therefore,
additional efforts to rank the component functions are needed, keeping
in mind that the same efforts may be recycled for the PDD approximation.
For this purpose, the authors propose two distinct ranking schemes:
(1) full ranking scheme and (2) a reduced ranking scheme, both exploiting
the global sensitivity index $G_{u}$ as a measure of the significance
of $y_{u}(\mathbf{X}_{u})$. However, since $G_{u}$ is estimated
by its $m_{u}$th-order polynomial approximation $\tilde{G}_{u,m_{u}}$,
any ranking system based on $\tilde{G}_{u,m_{u}}$, where $m_{u}$
is finite, may be in a flux and should hence be carefully interpreted.
This implies that a ranking scheme resulting from $\tilde{G}_{u,m_{u}}$,
whether full or reduced, must be iterated for increasing values of
$m_{u}$ until the ranking scheme converges according to a specified
criterion. In the full ranking scheme, all $|u|$-variate component
functions are re-ranked from scratch for each increment of $m_{u}$
until a converged ranking scheme emerges. Consequently, the full ranking
scheme affords any component function to contribute to the resultant
PDD approximation, provided that the criterion $\tilde{G}_{u,m_{u}}>\epsilon_{1}$
is satisfied only at convergence. In contrast, a subset of $|u|$-variate
component functions, determined from the previous ranking results and
truncations set by the tolerance criterion, are re-ranked for each
increment of $m_{u}$ in the reduced ranking scheme until convergence
is achieved. Therefore, for a component function from the reduced
ranking scheme to contribute to the resultant PDD approximation, the
criterion $\tilde{G}_{u,m_{u}}>\epsilon_{1}$ must be satisfied at
all ranking iterations including the converged one. Therefore, the
full ranking scheme is meticulous, but it is also exhaustive,
rapidly becoming inefficient or impractical when applied to high-dimensional
stochastic responses. The reduced ranking scheme, obtained less rigorously
than the former, is highly efficient and is ideal for solving industrial-scale
high-dimensional problems. A ranking system obtained at $m_{u}=m$,
$2\le m<\infty$, for all $|u|$-variate component functions is considered
to be converged if the ranking discrepancy ratio, defined as the ratio
of the number of ranked positions changed when $m_{u}$ increases
from $m-1$ to $m$ to the number of component functions ranked at
$m_{u}=m-1$, does not exceed the ranking tolerance $0\le\epsilon_{3}\le1$.
The number of component functions ranked in the full ranking scheme
is $L_{|u|}$, the total number of $|u|$-variate component functions,
and is the same for any $m_{u}$ or function $y$. In contrast, the
number of component functions ranked in the reduced ranking scheme,
which is equal to or less than $L_{|u|}$, depends on $m_{u}$, $y$,
and $\epsilon_{1}$. Both ranking schemes are described in Algorithm
\ref{ranking-algorithm}.

\begin{algorithm*}[!th]
\caption{\label{ranking-algorithm}Ranking of component functions.}
\begin{algorithmic}
\item sort $y_{v^{\left(l\right)}}(\boldsymbol{x}_{v^{\left(l\right)}})$: $l=1,\ldots,L$; $l=1$ for largest $\tilde{G}_{v,m_v}$ \Comment [$L=N$ for full ranking, or when $m_v=1$]
\end{algorithmic}
{\bf Truncation for reduced ranking:}
\begin{algorithmic}
\item $l \gets 1$
\While {$\tilde{G}_{v^{\left( l \right)},m_{v^{\left( l \right)}}} > \epsilon_1$} \Comment [truncating the ranking when adaptivity condition $\tilde{G}_{v^{\left( l \right)},m_{v^{\left( l \right)}}} > \epsilon_1$ fails]
\item \hspace {5 mm} $L \gets l$
\item \hspace {5 mm} $l \gets l+1$
\EndWhile
\end{algorithmic}
\end{algorithm*}

\begin{figure*}[tbph]
\begin{centering}
\includegraphics[scale=0.75]{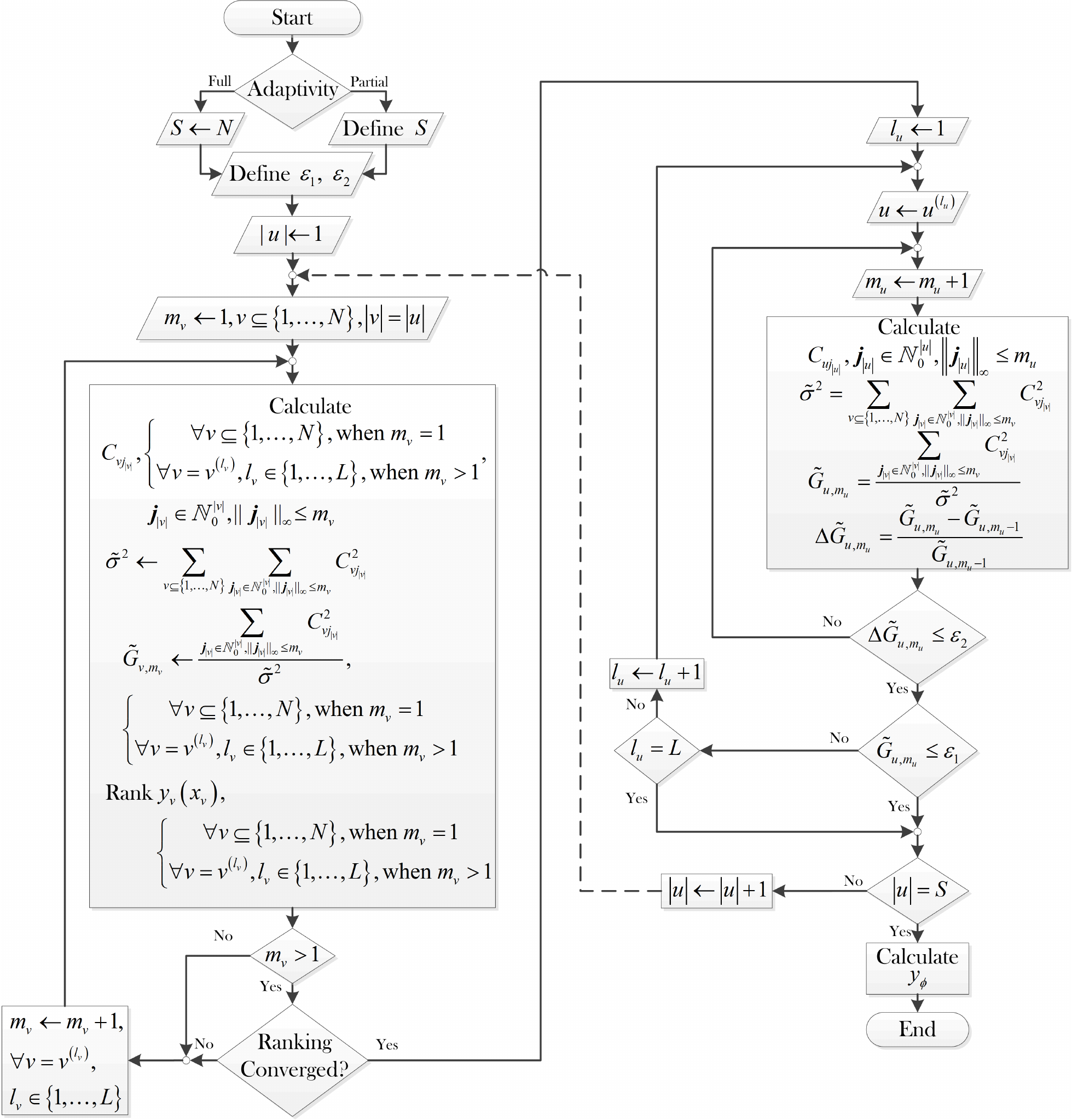}
\par\end{centering}
 \caption{\label{flowchart}A flowchart for constructing an adaptive-sparse
polynomial dimensional decomposition.}
\end{figure*}

\subsubsection{Computational Effort}

For uncertainty quantification, the computational effort is commonly
determined by the total number of original function evaluations. Consequently,
the efforts required by the proposed methods are proportional to the
total numbers of the PDD expansion coefficients retained in the concomitant
approximations and depend on the numerical techniques used to calculate
the coefficients. The numerical evaluation of the expansion coefficients
are discussed in Section 4.

The numbers of coefficients by the fully and partially adaptive-sparse
PDD methods are
\begin{equation}
\begin{array}{lcl}
\bar{K} & = & 1+{\displaystyle \sum_{\emptyset\ne u\subseteq\{1,\cdots,N\}}}\:{\displaystyle \sum_{m_{u}=1}^{\infty}}\:{\displaystyle \sum_{{\textstyle {\left\Vert \mathbf{j}_{|u|}\right\Vert _{\infty}=m_{u},\, j_{1},\cdots,j_{|u|}\neq0\atop \tilde{G}_{u,m_{u}}>\epsilon_{1},\Delta\tilde{G}_{u,m_{u}}>\epsilon_{2}}}}}1\\
 & = & 1+{\displaystyle \sum_{\emptyset\ne u\subseteq\{1,\cdots,N\}}}\:{\displaystyle \sum_{m_{u}=1}^{\infty}}\:{\displaystyle \sum_{\tilde{G}_{u,m_{u}}>\epsilon_{1},\Delta\tilde{G}_{u,m_{u}}>\epsilon_{2}}}\left[m_{u}^{\left|u\right|}-\left(m_{u}-1\right)^{\left|u\right|}\right]
\end{array}\label{sr16}
\end{equation}
and
\begin{equation}
\begin{array}{lcl}
\bar{K}_{S} & = & 1+{\displaystyle \sum_{{\textstyle {\emptyset\ne u\subseteq\{1,\cdots,N\}\atop 1\le|u|\le S}}}}\:{\displaystyle \sum_{m_{u}=1}^{\infty}}\:{\displaystyle \sum_{{\textstyle {\left\Vert \mathbf{j}_{|u|}\right\Vert _{\infty}=m_{u},\, j_{1},\cdots,j_{|u|}\neq0\atop \tilde{G}_{u,m_{u}}>\epsilon_{1},\Delta\tilde{G}_{u,m_{u}}>\epsilon_{2}}}}}1\\
 & = & 1+{\displaystyle \sum_{{\textstyle {\emptyset\ne u\subseteq\{1,\cdots,N\}\atop 1\le|u|\le S}}}}\:{\displaystyle \sum_{m_{u}=1}^{\infty}}\:{\displaystyle \sum_{\tilde{G}_{u,m_{u}}>\epsilon_{1},\Delta\tilde{G}_{u,m_{u}}>\epsilon_{2}}}\left[m_{u}^{\left|u\right|}-\left(m_{u}-1\right)^{\left|u\right|}\right],
\end{array}\label{sr17}
\end{equation}
respectively. It is elementary to show that $\bar{K}_{S}\le\bar{K}$
when $S\le N$ for identical tolerances, as expected, with equality
when $S=N$. Therefore, a partially adaptive-sparse PDD method in
general is more economical than the fully adaptive-sparse PDD method.

What can be inferred from the numbers of coefficients required by
a partially adaptive-sparse PDD method and the existing truncated
PDD method? The following two results, Proposition \ref{pro5} and \ref{pro6}, provide
some insights when the tolerances vanish and when the largest orders
of polynomials are identical.

\begin{pro}
\label{pro5}
If $\epsilon_{1}\to0$, and $\epsilon_{2}\to0$,
then $\bar{K}_{S}\rightarrow\tilde{K}_{S,m}$ as $m\to\infty$.
\end{pro}

\begin{proof}
From Equation \eqref{sr17},
\begin{equation}
\begin{array}{lcl}
{\displaystyle \lim_{\epsilon_{1}\to0 \atop \epsilon_{2}\to0}}\:\bar{K}_{S} & = & 1+{\displaystyle \sum_{{\textstyle {\emptyset\ne u\subseteq\{1,\cdots,N\}\atop 1\le|u|\le S}}}}\:{\displaystyle \sum_{m_{u}=1}^{\infty}}\:{\displaystyle \sum_{{\textstyle {\left\Vert \mathbf{j}_{|u|}\right\Vert _{\infty}=m_{u}\atop j_{1},\cdots,j_{|u|}\neq0}}}}1\\
 & = & 1+{\displaystyle \sum_{{\textstyle {\emptyset\ne u\subseteq\{1,\cdots,N\}\atop 1\le|u|\le S}}}}\:{\displaystyle \sum_{{\textstyle {\mathbf{j}_{|u|}\in\mathbb{N}_{0}^{|u|}\atop j_{1},\cdots,j_{|u|}\neq0}}}}1\\
 & = & {\displaystyle \lim_{m\to\infty}}\left[1+{\displaystyle \sum_{{\textstyle {\emptyset\ne u\subseteq\{1,\cdots,N\}\atop 1\le|u|\le S}}}}\:{\displaystyle \sum_{{\textstyle {\mathbf{j}_{|u|}\in\mathbb{N}_{0}^{|u|},\left\Vert \mathbf{j}_{|u|}\right\Vert _{\infty}\le m\atop j_{1},\cdots,j_{|u|}\neq0}}}}1\right]\\
 & = & {\displaystyle \lim_{m\to\infty}}\left[{\displaystyle \sum_{k=0}^{S}}\binom{N}{k}m^{k}\right]\\
 & = & {\displaystyle \lim_{m\to\infty}}\tilde{K}_{S,m},
\end{array}\label{sr17b}
\end{equation}
where the last line follows from Equation \eqref{eq:K-tilda-Sm}.
\end{proof}

\begin{pro}
\label{pro6}
If
\begin{equation}
m_{\max}=\max_{{\textstyle {\emptyset\ne u\subseteq\{1,\cdots,N\},1\le|u|\le S\atop \tilde{G}_{u,m_{u}}>\epsilon_{1},\Delta\tilde{G}_{u,m_{u}}>\epsilon_{2}}}}m_{u}<\infty\label{sr17c}
\end{equation}
is the largest order of polynomial expansion for any component function
$y_{u}(\mathbf{X}_{u})$, $\emptyset\ne u\subseteq\{1,\cdots,N\}$,
$1\le|u|\le S$, such that $\tilde{G}_{u,m_{u}}>\epsilon_{1},\Delta\tilde{G}_{u,m_{u}}>\epsilon_{2}$,
then $\bar{K}_{S}\le\tilde{K}_{S,m_{\max}}$.
\end{pro}

\begin{proof}
From Equation \eqref{sr17},
\begin{equation}
\begin{array}{lcl}
\bar{K}_{S} & = & 1+{\displaystyle \sum_{{\textstyle {\emptyset\ne u\subseteq\{1,\cdots,N\}\atop 1\le|u|\le S}}}}\:{\displaystyle \sum_{m_{u}=1}^{\infty}}\:{\displaystyle \sum_{{\textstyle {\left\Vert \mathbf{j}_{|u|}\right\Vert _{\infty}=m_{u},\, j_{1},\cdots,j_{|u|}\neq0\atop \tilde{G}_{u,m_{u}}>\epsilon_{1},\Delta\tilde{G}_{u,m_{u}}>\epsilon_{2}}}}}1\\
 & \le & 1+{\displaystyle \sum_{{\textstyle {\emptyset\ne u\subseteq\{1,\cdots,N\}\atop 1\le|u|\le S}}}}\:{\displaystyle \sum_{{\textstyle {\mathbf{j}_{|u|}\in\mathbb{N}_{0}^{|u|},\left\Vert \mathbf{j}_{|u|}\right\Vert _{\infty}\le m_{\max}\atop j_{1},\cdots,j_{|u|}\neq0}}}}1\\
 & = & {\displaystyle \sum_{k=0}^{S}}\binom{N}{k}m_{\max}^{k}\\
 & = & \tilde{K}_{S,m_{\max}},
\end{array}\label{sr17d}
\end{equation}
where the last line follows from Equation \eqref{eq:K-tilda-Sm}.
\end{proof}

According to Proposition \ref{pro6}, the partially adaptive-sparse PDD approximation
for non-trivial tolerances should be computationally more efficient
than the truncated PDD approximation. However, the computational efforts by both approximations depend on the numerical technique employed to estimate the associated expansion coefficients. For instance, suppose that a full-grid dimension-reduction integration with its own truncation $R = S$, to be explained in Section 4, is applied to calculate all $\tilde{K}_{S,m_{\max}}$ expansion coefficients to achieve the accuracy of an $S$-variate, $m_{\max}$th-order PDD approximation. Then the requisite number of function evaluations is $S$th-order polynomial with respect to $N$, the size of the stochastic problem. The partially adaptive-sparse PDD approximation, while retaining a similar accuracy, is expected to markedly reduce the number of function calls. This issue will be further explored in Example 3 of the Numerical Examples section.

\section{Calculation of Expansion Coefficients}

The determination of the expansion coefficients $y_{\emptyset}$ and
$C_{u\mathbf{j}_{|u|}}$ in Equations \eqref{1b} and \eqref{7} involves
various $N$-dimensional integrals over $\mathbb{R}^{N}$. For large
$N$, a full numerical integration employing an $N$-dimensional tensor
product of a univariate quadrature formula is computationally prohibitive
and is, therefore, ruled out. Two new alternative numerical techniques
are proposed to estimate the coefficients accurately and efficiently.

\subsection{Dimension-Reduction Integration}

The dimension-reduction integration, developed by Xu and Rahman \cite{xu04},
entails approximating a high-dimensional integral of interest by a
finite sum of lower-dimensional integrations. For calculating the
expansion coefficients $y_{\emptyset}$ and $C_{u\mathbf{j}_{|u|}}$,
this is accomplished by replacing the $N$-variate function $y$ in
Equations \eqref{1b} and \eqref{7} with an $R$-variate RDD approximation
at a chosen reference point, where $R\le N$ \cite{xu04,xu05}. The result
is a reduced integration scheme, requiring evaluations of at most
$R$-dimensional integrals.

Given a reference point $\mathbf{c}=(c_{1},\cdots,c_{N})\in\mathbb{R}^{N}$
and RDD component functions $w_{\emptyset}$ and $w_{u}(\mathbf{X}_{u};\mathbf{c})$
described by Equations \eqref{5b} and \eqref{5c}, let $\hat{y}_{R}(\mathbf{X};\mathbf{c})$
(Equation \eqref{sr2}) denote an $R$-variate RDD approximation of $y(\mathbf{X})$. Replacing
$y(\mathbf{x})$ in Equations \eqref{1b} and \eqref{7} with $\hat{y}_{R}(\mathbf{x};\mathbf{c})$, the coefficients $y_{\emptyset}$ and $C_{u\mathbf{j}_{|u|}}$
are estimated from \cite{xu04}
\begin{equation}
y_{\emptyset}\cong{\displaystyle \sum_{i=0}^{R}}(-1)^{i}{N-R+i-1 \choose i}\sum_{{\textstyle {v\subseteq\{1,\cdots,N\}\atop |v|=R-i}}}\!\int_{\mathbb{R}^{|v|}}{\displaystyle y(\mathbf{x}_{v},\mathbf{c}_{-v})}f_{\mathbf{X}_{v}}(\mathbf{x}_{v})d\mathbf{x}_{v}\label{sr18}
\end{equation}
 and
\begin{equation}
\begin{split}
C_{u\mathbf{j}_{|u|}}&\cong{\displaystyle \sum_{i=0}^{R}}(-1)^{i}{N-R+i-1 \choose i}\\
&\sum_{{\textstyle {v\subseteq\{1,\cdots,N\}\atop |v|=R-i,u\subseteq v}}}\!\int_{\mathbb{R}^{|v|}}{\displaystyle y(\mathbf{x}_{v},\mathbf{c}_{-v})\psi_{u\mathbf{j}_{|u|}}(\mathbf{x}_{u})}f_{\mathbf{X}_{v}}(\mathbf{x}_{v})d\mathbf{x}_{v},\label{sr19}
\end{split}
\end{equation}
respectively, requiring evaluation of at most $R$-dimensional integrals.
The reduced integration facilitates calculation of the coefficients
approaching their exact values as $R\to N$, and is significantly
more efficient than performing one $N$-dimensional integration, particularly
when $R\ll N$. Hence, the computational effort is significantly decreased
using the dimension-reduction integration. For instance, when $R=1$
or $2$, Equations \eqref{sr18} and \eqref{sr19} involve one-, or
at most, two-dimensional integrations, respectively. Nonetheless,
numerical integrations are still required for performing various $|v|$-dimensional
integrals over $\mathbb{R}^{|v|}$, where $0\le|v|\le R$. When $R>1$,
the multivariate integrations involved can be conducted using full-
or sparse-grids, as follows.

\subsubsection{Full-Grid Integration}

The full-grid dimension-reduction integration entails constructing
a tensor product of the underlying univariate quadrature rules. For
a given $v\subseteq\{1,\cdots,N\}$, $1<|v|\le R$, let $v=\{i_{1},\cdots i_{|v|}\}$,
where $1\le i_{1}<\cdots<i_{|v|}\le N$. Denote by $\{x_{i_{p}}^{(1)},\cdots,x_{i_{p}}^{(n_{v})}\}\subset\mathbb{R}$
a set of integration points of $x_{i_{p}}$ and by $\{w_{i_{p}}^{(1)},\cdots,w_{i_{p}}^{(n_{v})}\}$
the associated weights generated from a chosen univariate quadrature
rule and a positive integer $n_{v}\in\mathbb{N}$. Denote by $P^{(n_{v})}=\times_{p=1}^{p=|v|}\{x_{i_{p}}^{(1)},\cdots,x_{i_{p}}^{(n_{v})}\}$
the rectangular grid consisting of all integration points generated
by the variables indexed by the elements of $v$. Then the coefficients
using dimension-reduction numerical integration with a full-grid are
approximated by
\begin{equation}
y_{\emptyset}\cong{\displaystyle \sum_{i=0}^{R}}(-1)^{i}{N-R+i-1 \choose i}\sum_{{\textstyle {v\subseteq\{1,\cdots,N\}\atop |v|=R-i}}}\!\sum_{\mathbf{k}_{|v|}\in P^{(n_{v})}}w^{(\mathbf{k}_{|v|})}y(\mathbf{x}_{v}^{(\mathbf{k}_{|v|})},\mathbf{c}_{-v}),\label{sr20}
\end{equation}
\begin{equation}
\begin{split}
C_{u\mathbf{j}_{|u|}}&\cong{\displaystyle \sum_{i=0}^{R}}(-1)^{i}{N-R+i-1 \choose i}\\
&\sum_{{\textstyle {v\subseteq\{1,\cdots,N\}\atop |v|=R-i,u\subseteq v}}}\!\sum_{\mathbf{k}_{|v|}\in P^{(n_{v})}}w^{(\mathbf{k}_{|v|})}y(\mathbf{x}_{v}^{(\mathbf{k}_{|v|})},\mathbf{c}_{-v})\psi_{u\mathbf{j}_{|u|}}(\mathbf{x}_{u}^{(\mathbf{k}_{|u|})}),\label{sr21}
\end{split}
\end{equation}
where $\mathbf{x}_{v}^{(\mathbf{k}_{|v|})}=\{x_{i_{1}}^{(k_{1})},\cdots,x_{i_{|v|}}^{(k_{|v|})}\}$
and $w^{(\mathbf{k}_{|v|})}=\prod_{p=1}^{p=|v|}w_{i_{p}}^{(k_{p})}$
is the product of integration weights generated by the variables indexed
by the elements of $v$. For independent coordinates of $\mathbf{X}$,
as assumed here, a univariate Gauss quadrature rule is commonly used,
where the integration points and associated weights depend on the
probability distribution of $X_{i}$. They are readily available,
for example, the Gauss-Hermite or Gauss-Legendre quadrature rule,
when $X_{i}$ follows Gaussian or uniform distribution \cite{gautschi04}. For
an arbitrary probability distribution of $X_{i}$, the Stieltjes procedure
\cite{gautschi04} can be employed to generate the measure-consistent Gauss
quadrature formulae \cite{gautschi04}. An $n_{v}$-point Gauss quadrature rule
exactly integrates a polynomial of total degree at most $2n_{v}-1$.

The calculation of $y_{\emptyset}$ and $C_{u\mathbf{j}_{|u|}}$ from
Equations \eqref{sr20} and \eqref{sr21} involves at most $R$-dimensional
tensor products of an $n_{v}$-point univariate quadrature rule, requiring
the following deterministic responses or function evaluations: $y(\mathbf{c})$,
$y(\mathbf{x}_{v}^{(\mathbf{j}_{|v|})},\mathbf{c}_{-v})$ for $i=0,\cdots,R$,
$v\subseteq\{1,\cdots,N\}$, $|v|=R-i$, and $\mathbf{j}_{|v|}\in P^{(n_{v})}$.
Accordingly, the total cost for estimating the PDD expansion coefficients
entails
\begin{equation}
L_{FG}={\displaystyle \sum_{i=0}^{R}}\sum_{{\textstyle {v\subseteq\{1,\cdots,N\}\atop |v|=R-i}}}n_{v}^{|v|}\label{sr22}
\end{equation}
function evaluations, encountering a computational complexity that
is $R$th-order polynomial $-$ for instance, linear or quadratic
when $R=1$ or $2$ $-$ with respect to the number of random variables
or integration points. For $R<N$, the technique alleviates the curse
of dimensionality to an extent determined by $R$.

\subsubsection{Sparse-Grid Integration}

Although the full-grid dimension-reduction integration has been successfully
applied to the calculation of the PDD expansion coefficients in the
past \cite{rahman08,rahman09,rahman10,rahman11}, it faces a major drawback when the polynomial order
$m_{u}$ for a PDD component function $y_{u}$ needs to be modulated
for adaptivity. As the value of $m_{u}$ is incremented by one, a
completely new set of integration points is generated by the univariate
Gauss quadrature rule, rendering all expensive function evaluations
on prior integration points as useless. Therefore, a nested Gauss
quadrature rule, such as the fully symmetric interpolatory rule,
that is capable of exploiting dimension-reduction integration is proposed.

\paragraph{Fully symmetric interpolatory rule}
The fully symmetric interpolatory (FSI) rules
developed by Genz and his associates \cite{genz86,genz96}, is a sparse-grid
integration technique for performing high-dimensional numerical integration.
Applying this rule to the $|v|$-dimensional integrations in Equations \eqref{sr18}
and \eqref{sr19}, the PDD expansion coefficients are approximated
by
\begin{equation}
\begin{split}
y_{\emptyset} & \cong{\displaystyle \sum_{i=0}^{R}}(-1)^{i}{N-R+i-1 \choose i}\sum_{{\textstyle {v\subseteq\{1,\cdots,N\}\atop |v|=R-i}}}\!\sum_{\mathbf{p}_{|v|}\in P^{\left(\tilde{n}_{v},|v|\right)}}w_{\mathbf{p}_{|v|}}\\
& \sum_{\mathbf{\mathbf{q}_{|v|}}\in\Pi_{\mathbf{p}_{|v|}}}\sum_{\mathbf{t}_{|v|}}y\left(t_{i_{1}}\alpha_{q_{i_{1}}},\cdots,t_{i_{|v|}}\alpha_{q_{i_{|v|}}},\mathbf{c}_{-v}\right),\label{sr23}
\end{split}
\end{equation}
\begin{equation}
\begin{split}
C_{u\mathbf{j}_{|u|}} & \cong{\displaystyle \sum_{i=0}^{R}}(-1)^{i}{N-R+i-1 \choose i}\sum_{{\textstyle {v\subseteq\{1,\cdots,N\}\atop |v|=R-i,u\subseteq v}}}\!\sum_{\mathbf{p}_{|v|}\in P^{\left(\tilde{n}_{v},|v|\right)}}w_{\mathbf{p}_{|v|}}\\
 &\!\!\!\!\!\!\!\! \sum_{\mathbf{\mathbf{q}_{|v|}}\in\Pi_{\mathbf{p}_{|v|}}}\!\!\! \sum_{\mathbf{t}_{|v|}}\!y\left(t_{i_{1}}\alpha_{q_{i_{1}}},\cdots,t_{i_{|v|}}\alpha_{q_{i_{|v|}}},\mathbf{c}_{-v}\right)\psi_{u\mathbf{j}_{|u|}}\left(t_{i_{1}}\alpha_{q_{i_{1}}},\cdots,t_{i_{|u|}}\alpha_{q_{i_{|u|}}}\right),\label{sr24}
\end{split}
\end{equation}
where $v=\{i_{1},\cdots i_{|v|}\}$, $\mathbf{t}_{|v|}=(t_{i_{1}},\cdots,t_{i_{|v|}})$,
$\mathbf{p}_{|v|}=(p_{i_{1}},\cdots,p_{i_{|v|}})$,
\begin{equation}
P^{\left(\tilde{n}_{v},|v|\right)}=\{\mathbf{p}_{|v|}:\tilde{n}_{v}\ge p_{i_{1}}\ge\cdots\ge p_{i_{|v|}}\ge0,\left\Vert \mathbf{p}_{|v|}\right\Vert \le \tilde{n}_{v}\}\label{sr25}
\end{equation}
with $\left\Vert \mathbf{p}_{|v|}\right\Vert :=\sum_{r=1}^{|v|}p_{i_{r}}$
is the set of all distinct $|v|$-partitions of the integers $0,1,\cdots,\tilde{n}_{v}$,
and $\Pi_{\mathbf{p}_{|v|}}$ is the set of all permutations of $\mathbf{p}_{|v|}$.
The innermost sum over $\mathbf{t}_{|v|}$ is taken over all of the
sign combinations that occur when $t_{i_{r}}=\pm1$ for those values
of $i_{r}$ with generators $\alpha_{q_{i_{r}}}\ne0$ \cite{genz96}. The
weight
\begin{equation}
w_{\mathbf{\mathbf{p}_{|v|}}}=2^{-K}\sum_{\left\Vert \mathbf{\mathbf{k}_{|v|}}\right\Vert \leqslant \tilde{n}_{v}-\left\Vert \mathbf{p}_{|v|}\right\Vert }{\displaystyle \prod_{r=1}^{|v|}}{\displaystyle \frac{a_{k_{i_{r}}+p_{i_{r}}}}{{\displaystyle \prod_{j=0,j\ne p_{i_{r}}}^{k_{i_{r}}+p_{i_{r}}}}\left(\alpha_{p_{i_{r}}}^{2}-\alpha_{j}^{2}\right)}},\label{sr26}
\end{equation}
where $K$ is the number of nonzero components in $\mathbf{p}_{|v|}$
and $a_{i}$ is a constant that depends on the probability measure
of $X_{i},$ for instance,
\begin{equation}
a_{i}=\frac{1}{\sqrt{2\pi}}\int_{\mathbb{R}}\exp\left(-\frac{\xi^{2}}{2}\right)\prod_{j=0}^{i-1}\left(\xi^{2}-\alpha_{j}^{2}\right)d\xi\label{sr27}
\end{equation}
for $i>0$ and $a_{0}=1$ when $X_{i}$ follows the standard Gaussian
distribution \cite{genz96}. An $\tilde{n}_{v}$-parameter FSI rule exactly integrates
a polynomial of degree at most $2\tilde{n}_{v}+1$.

\paragraph{Extended fully symmetric interpolatory rule}
The number of function evaluations by the original FSI rule \cite{genz86}
increases rapidly as $|v|$ and $\tilde{n}_{v}$ increase. To enhance the
efficiency, Genz and Keister \cite{genz96} proposed an extended FSI rule
in which the function evaluations are significantly reduced if the
generator set is chosen such that some of the weights $w_{\mathbf{\mathbf{p}_{|v|}}}$
are \emph{zero}. The pivotal step in constructing such FSI rule is to extend
a $(2\beta+1)$-point Gauss-Hermite quadrature rule by adding $2\gamma$
points or generators $\pm\alpha_{\beta+1},\pm\alpha_{\beta+2},\ldots,\pm\alpha_{\beta+\gamma}$
with the objective of maximizing the degree of polynomial exactness
of the extended rule, where $\beta\in\mathbb{N}$ and $\gamma\in\mathbb{N}$.
Genz and Keister \cite{genz96} presented a special case of initiating the
FSI rule from the univariate Gauss-Hermite rule over the interval
$\left(-\infty,\infty\right)$. The additional generators in this
case are determined as roots of the monic polynomial $\zeta^{2\gamma}+t_{\gamma-1}\zeta^{2\gamma-1}+\cdots+t_{0}$,
where the coefficients $t_{\gamma-1},\cdots,t_{0}$ are obtained by
invoking the condition
\begin{equation}
\frac{1}{\sqrt{2\pi}}\int_{\mathbb{R}}\exp\left(-\frac{\xi^{2}}{2}\right)\prod_{j=0}^{\beta}\xi^{2b}\left(\xi^{2}-\alpha_{j}^{2}\right)d\xi=0,\label{sr28}
\end{equation}
where $\gamma>\beta$. A new set of generators is propagated based
on the prior rule and, therefore, as the polynomial degree of exactness
of the rule increases, all the previous points and the expensive function
evaluations over those points are preserved. A remarkable feature
of the extended FSI rule is that the choice of generators is such
that some of the weights \textit{$w_{\mathbf{\mathbf{p}_{|v|}}}=0$}
in each step of the extension \cite{genz96}, thus eliminating the need
for function evaluations at the integration points corresponding to
\emph{zero} weights, making the extended FSI rule significantly more efficient
than its earlier version.

Since RDD is tied with the reference point, the dimension-reduction integration, whether full-grid or sparse-grid, to calculate the PDD expansion coefficients also depends on $\mathbf{c}$. However, from past experience \cite{zabaras10, rahman08, rahman09, xu04, karniadakis12}, very accurate estimates of the expansion coefficients were obtained when $\mathbf{c}$ is selected as the mean value of $\mathbf{X}$. A more rigorous approach entails finding an optimal reference point, but it will require additional function evaluations and hence may render the dimension-reduction technique impractical for solving high-dimensional problems.

\subsubsection{Integration Points}

The number of integration points determines the computational expense
incurred in calculating the PDD expansion coefficients. Therefore, it
is instructive to compare the numbers of points required by full-
or sparse-grid dimension-reduction integrations. To do so, consider
the efforts in performing a $|v|$-dimensional integration in Equation
\eqref{sr18} or \eqref{sr19} over the interval $\left(-\infty,\infty\right)$
by three different numerical techniques: (1) the full-grid integration
technique; (2) the sparse-grid integration technique using the extended
FSI rule; and (3) the sparse-grid integration technique using Smolyak's
algorithm \cite{novak99}. The Smolyak's algorithm is included because
it is commonly used as a preferred sparse-grid numerical technique
for approximating high-dimensional integrals. Define an integer $l\in\mathbb{N}$
such that all three techniques can exactly integrate a polynomial
function of total degree $2l-1$. For instance, when
$l=3$, all three techniques exactly integrate a quintic polynomial.
Figure \ref{SparseFig1} presents a comparison of the total numbers
of integration points in a two-dimensional grid, that is, when $|v|=2$,
for $l$ ranging from one through five by the three distinct
multivariate integration techniques. Each plot illustrates two numbers: the first number indicates the number of integration points required at the given value of $l$; the second number, inside the parenthesis, indicates the total number of cumulative integration points added up to the value of $l$. It is imperative to add the integration points from all the previous values of $l$ as it reflects the total number of function evaluations required in an adaptive algorithm. For the full-grid integration, the two numbers are different for all $l>1$, indicating a lack of nesting of the integration points. Whereas in the sparse-grid with extended FSI rule, the two numbers are equal for all $l$, reflecting the fully nested integration points in this rule. As $l$ increments,
a completely new set of points is introduced in the full-grid integration,
rendering the prior points useless. However, for fairness in comparison,
it is necessary to consider all points from prior values of $l$ as the expensive
function evaluations have already been performed. Therefore, Figure
\ref{SparseFig1} captures the cumulative numbers of integration points
as $l$ increases steadily. For values of $l$ up to two, all three techniques
require the same number of integration points. However, differences
in the numbers of points start to appear in favor of the extended
FSI rule when $l$ exceeds two, making it the clear favorite
among all three techniques for high-order numerical integration. The
Smolyak's algorithm, which is not nested, is the least efficient of the three techniques. The extended FSI rule, in contrast, is fully nested, establishing a principal advantage over Smolyak's algorithm for adaptive
numerical integration.

\begin{figure*}[tbph]
\begin{centering}
\includegraphics[scale=0.8]{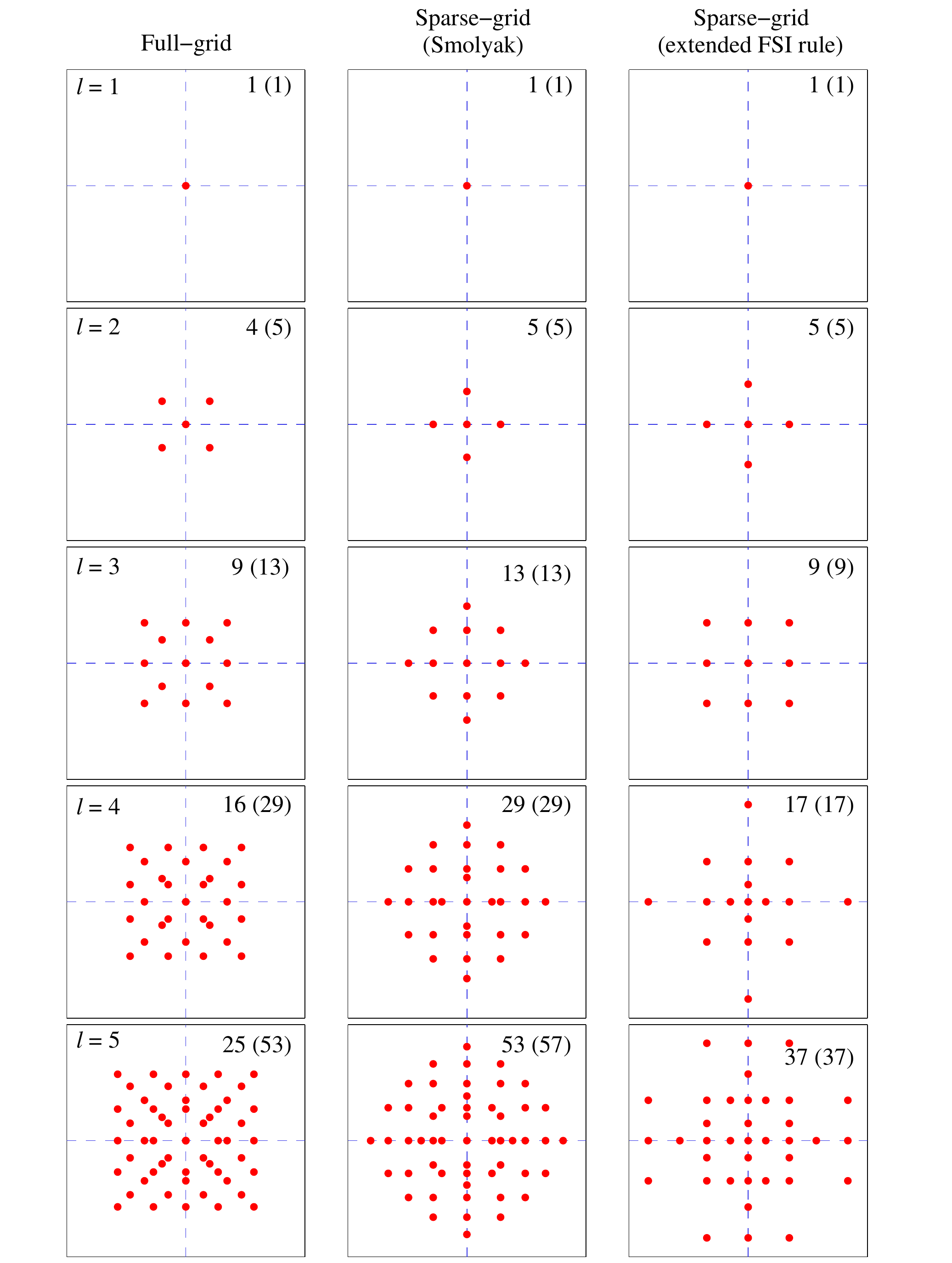}
\par\end{centering}
 \caption{\label{SparseFig1}Gauss-Hermite integration points
in a two-dimensional grid by the full-grid technique, sparse-grid
with the extended FSI rule, and sparse-grid with Smolyak's algorithm
for various levels. Note: each grid is plotted over a square with axes ranging from $-5$ to $5$.}
\end{figure*}

Table \ref{IntegPointsTable} lists the number of integration points required at the integration rule corresponding to a given value of $l$, for $2\leq |v| \leq 10$ and $2\leq l \leq5$. It is important to note that the number of integration points listed is not cumulative. It appears that for higher-dimensional integrations, that is, for $|v|>2$, the extended FSI rule is markedly more efficient than full-grid or other sparse-grid techniques even for the non-cumulative points. The efficiency of extended FSI rule is more pronounced for cumulative number of integration points. For further details, the reader is referred to the work of Genz and Keister
\cite{genz96}, who examined the extended FSI rule for dimensions up to
20.
\begin{table*}[tbph]
\centering
\begin{threeparttable}
\caption{\label{IntegPointsTable}Number of integration points in various $|v|$-dimensional integration techniques, each technique exactly integrates polynomials of total order $2l-1$.}
\begin{centering}
\begin{tabular}{crrrrrrrrr}
\hline
 & \multicolumn{9}{c}{$|v|$}\tabularnewline
\cline{2-10}
$l$ & 2 & 3 & 4 & 5 & 6 & 7 & 8 & 9 & 10\tabularnewline
\hline
 & \multicolumn{9}{c}{(a) Full-grid}\tabularnewline
\cline{2-10}
2 & 4 & 8 & 16 & 32 & 64 & 128 & 256 & 512 & 1024\tabularnewline
3 & 9 & 27 & 81 & 243 & 729 & 2187 & 6561 & 19683 & 59049\tabularnewline
4 & 16 & 64 & 256 & 1024 & 4096 & 16384 & 65536 & 262144 & 1048576\tabularnewline
5 & 25 & 125 & 625 & 3125 & 15625 & 78125 & 390625 & 1953125 & 9765625\tabularnewline
\cline{1-10}
 & \multicolumn{9}{c}{(b) Sparse-grid (Smolyak)}\tabularnewline
\cline{2-10}
2 & 5 & 7 & 9 & 11 & 13 & 15 & 17 & 19 & 21\tabularnewline
3 & 13 & 25 & 41 & 61 & 85 & 113 & 145 & 181 & 221\tabularnewline
4 & 29 & 69 & 137 & 241 & 389 & 589 & 849 & 1177 & 1581\tabularnewline
5 & 53 & 165 & 385 & 781 & 1433 & 2437 & 3905 & 5965 & 8761\tabularnewline
\cline{1-10}
 & \multicolumn{9}{c}{(c) Sparse-grid (extended FSI rule)}\tabularnewline
\cline{2-10}
2 & 5 & 7 & 9 & 11 & 13 & 15 & 17 & 19 & 21\tabularnewline
3 & 9 & 19 & 33 & 51 & 73 & 99 & 129 & 163 & 201\tabularnewline
4 & 17 & 39 & 81 & 151 & 257 & 407 & 609 & 871 & 1201\tabularnewline
5 & 37 & 93 & 201 & 401 & 749 & 1317 & 2193 & 3481 & 5301\tabularnewline
\hline
\end{tabular}
\par\end{centering}
\end{threeparttable}
\end{table*}

\subsection{Quasi Monte Carlo Simulation}

The basic idea of the quasi MCS is to replace the random or pseudo-random
samples in crude MCS by well-chosen deterministic samples that are
highly equidistributed \cite{niederreiter92}. The qausi Monte Carlo samples are
often selected from a low-discrepancy sequence \cite{faure81,halton60,niederreiter92,sobol67} or by a
lattice rule \cite{sloan94} to minimize the integration errors. The estimation
of the PDD expansion coefficients, which are high-dimensional integrals,
comprises three simple steps: (1) generate a low-discrepancy point
set $\mathcal{P}_{L}:=\{\mathbf{u}^{(k)}\in[0,1]^{N},\; k=1,\cdots,L\}$
of size $L\in\mathbb{N}$; (2) map each sample from $\mathcal{P}_{L}$
to the sample $\mathbf{x}^{(k)}\in\mathbb{R}^{N}$ following the probability
measure of the random input $\mathbf{X}$; and (3) approximate the
coefficients by
\begin{equation}
y_{\emptyset}\cong\frac{1}{L}\sum_{k=1}^{L}y\left(\mathbf{x}^{(k)}\right),\label{sr29}
\end{equation}
\begin{equation}
C_{u\mathbf{j}_{|u|}}\cong\frac{1}{L}\sum_{k=1}^{L}y\left(\mathbf{x}^{(k)}\right)\psi_{u\mathbf{j}_{|u|}}\left(\mathbf{x}_{u}^{(k)}\right).\label{sr30}
\end{equation}
The well-known Koksma\textendash{}Hlawka inequality reveals that the
error committed by the quasi MCS is bounded by the variation of the
integrand in the sense of Hardy and Krause and the star-discrepancy,
a measure of uniformity, of the point set $\mathcal{P}_{L}$ \cite{niederreiter92}.
Therefore, constructing a point set with star-discrepancy as small
as possible and seeking variance reduction of the integrand are vital
for the success of the quasi MCS. It should be mentioned here that
many authors, including Halton \cite{halton60}, Faure \cite{faure81}, Niederreiter
\cite{niederreiter92}, and Sobol \cite{sobol67}, and Wang \cite{wang00}, have extensively
studied how to generate the best low-discrepancy point sets and to
facilitate variance reduction. For a bounded variation of the integrand,
the quasi MCS has a theoretical error bound $O(L^{-1}(\log L)^{N}$
compared with the probabilistic error bound $O(L^{-1/2})$ of crude
MCS, indicating significantly faster convergence of the quasi MCS
than crude MCS.

The two proposed techniques for calculating the PDD coefficients represent two broad categories of numerical integration: the quadrature-based methods and the sampling-based methods. However, the calculation of PDD coefficients is not limited to only these two techniques. Furthermore, the relative accuracy or efficiency of one technique over the other depends on the dimension of the stochastic problem. For hundreds or thousands of random variables, a sampling-based technique is generally preferred over a quadrature-based technique, as the former is relatively insensitive to the problem size.

\section{Numerical Examples}

\label{examples}

Three numerical examples are put forward to illustrate the adaptive-sparse
PDD methods developed in calculating various probabilistic characteristics
of random mathematical functions and random eigensolutions of stochastic
dynamical systems. A principal objective is to compare the performance of the proposed adaptive-sparse PDD methods with that of the existing truncated PDD method. Readers interested in contrasting the truncated PDD method with the PCE \cite{ghanem91} and other classical methods are referred to the authors' prior work \cite{rahman07,rahman08,rahman09,rahman11}.

Classical Legendre polynomials were used to define
the orthonormal polynomials in Example 1, and all expansion coefficients
were determined analytically. In Examples 2 and 3, all original random
variables were transformed into standard Gaussian random variables,
facilitating the use of classical Hermite orthonormal polynomials
as bases. Since Example 2 consists of only nine input random variables,
the expansion coefficients were estimated using a nine-dimensional
tensor product of five-point univariate Gauss-Hermite quadrature rule.
The expansion coefficients in Example 3 were approximated by both
the full-grid dimension-reduction integration and sparse-grid dimension-reduction
integration with the extended FSI rule, where $R=S$ and $\mathbf{c}$ is the mean of $\mathbf{X}$. The sample sizes for crude
MCS in Example 2 is $10^{6}$. In Example 3, the sample
size for crude MCS is $50,000$, and for the embedded MCS, whether the
truncated or adaptive-sparse PDD method, the sample size is $10^{6}$.

\subsection{Example 1: A Polynomial Function}

Consider the polynomial function
\begin{equation}
y\left(\mathbf{X}\right)=\frac{{\displaystyle \prod_{i=1}^{N}\left(\frac{3}{i}X_{i}^{5}+1\right)}}{\mathbb{E}{\displaystyle \left[\prod_{i=1}^{N}\left(\frac{3}{i}X_{i}^{5}+1\right)\right]}}, \label{ex1}
\end{equation}
where $X_{i}$, $i=1,\cdots,N$, are independent and identical random
variables, each following the standard uniform distribution over $\left[0,1\right]$.
Since the coefficient of $X_{i}^{5}$ is inversely proportional to
$i$, the first and last random variables have the largest and least
influence on $y$. From elementary calculations, the exact mean and
variance of $y$ are 1 and
\begin{equation}
{\displaystyle \prod_{i=1}^{N}}\left({\displaystyle \frac{25}{11\left(1+2i\right)^{2}}+1}\right) - 1, \label{var1}
\end{equation}
respectively. All PDD expansion coefficients were calculated analytically.
Therefore, the ranking of component functions was performed once and
for all, avoiding any role of the ranking scheme in this particular
example. The numerical results that follow in the remainder of this subsection
were obtained for $N=5$.

Figure \ref{ex1-fig1} shows how the relative errors, defined as the
ratio of the absolute difference between the exact (Expression \eqref{var1})
and approximate (Equation \eqref{11}) variances of $y$ to the exact
variance, committed by $S$-variate, $m$-th order PDD approximations
vary with increasing polynomial order $m$. The five plots of univariate
($S=1$) to pentavariate ($S=5$) PDD approximations clearly show
that the error drops monotonically with respect to $m$ regardless
of $S$. When $m$ reaches five, the pentavariate PDD approximation
does not perpetrate any error, producing the exact variance of $y$
as expected. In contrast, the relative errors in variance caused by
fully adaptive-sparse PDD approximations (Equation \eqref{sr13}), also
illustrated in Figure \ref{ex1-fig1} for specified tolerances ranging
from $10^{-9}$ to $10^{-3}$, do not rely on $S$ or $m$, as the
degrees of interaction and polynomial orders are adaptively modulated
in the concomitant approximations. The adaptive-sparse PDD approximations
with tolerances equal to $10^{-3}$ and $10^{-4}$ yield relative
errors in variance marginally higher than the tolerance values; however,
the relative errors achieved are invariably smaller than all respective
values of the subsequent tolerances, demonstrating a one-to-one relationship
between the tolerance and relative error attained in calculating the
variance. As the tolerance decreases, so does the relative error. While a traditional truncated PDD approximation provides
options to increase the values of $S$ and/or $m$ for reducing the
relative error, the user remains blinded to the outcome of such an
action. The adaptive-sparse PDD method, in the form of tolerances,
provides a direct key to regulate the accuracy of the resultant approximation.

Figure \ref{ex1-fig2} displays the increase in number of PDD expansion
coefficients required by truncated (Equation \eqref{eq:K-tilda-Sm})
and fully adaptive-sparse (Equation \eqref{sr16}) PDD methods in order
to achieve a user-specified relative error in variance ranging from
$10^{-1}$ to $10^{-12}$. The relative error decreases from left
to right along the horizontal axis of the plot. The plot of the truncated
PDD approximation is generated by trial-and-error, increasing the value
of either $S$ or $m$ until the desired relative error is achieved
and then counting the total number of coefficients required to attain
that relative error. For obtaining the plot of the adaptive-sparse
PDD approximation, the tolerance values were reduced monotonically
and the corresponding total number of coefficients was noted for each
value of relative error. Ignoring the two lowest relative errors,
the comparison of the plots from these two methods clearly demonstrates
how the adaptive-sparse PDD method requires fewer expansion coefficients
than the truncated PDD method to achieve the desired level of relative
error. While the adaptive-sparse PDD method intelligently calculates
only those coefficients that are making significant contribution to
the variance, the truncated PDD method ends up calculating more coefficients
than required. Therefore, the adaptive-sparse PDD approximation represents
a more scientific and efficient method than the truncated PDD methods.

\begin{figure}[tbph]
\begin{centering}
\includegraphics[bb=10bp 10bp 650bp 510bp,clip,scale=0.49]{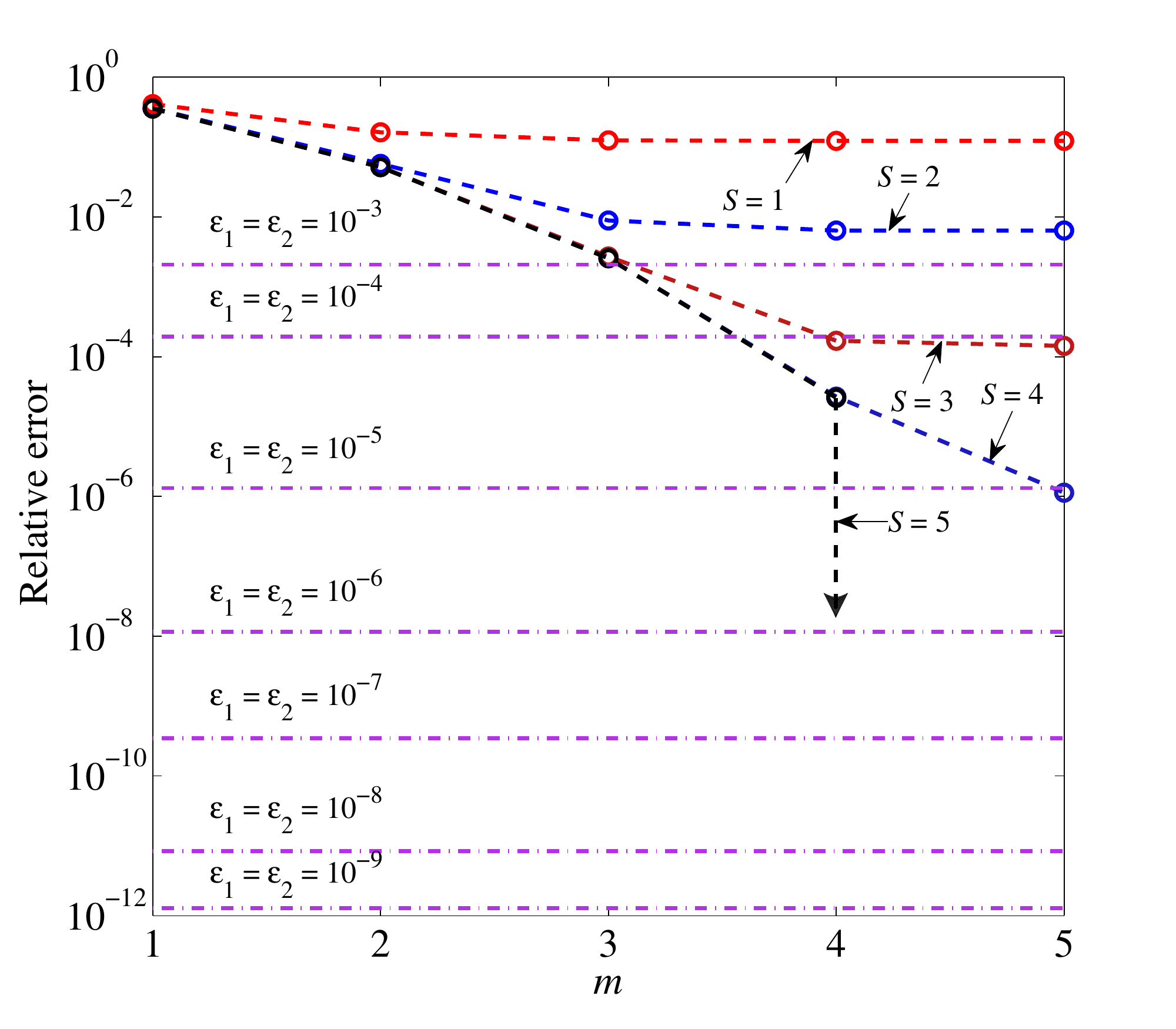}
\par\end{centering}
\caption{\label{ex1-fig1}Relative error in calculating the variance of a mathematical
function by fully adaptive-sparse and truncated PDD methods (Example
1).}
\end{figure}

\begin{figure}[tbph]
\begin{centering}
\includegraphics[bb=16bp 15bp 700bp 480bp,clip,scale=0.49]{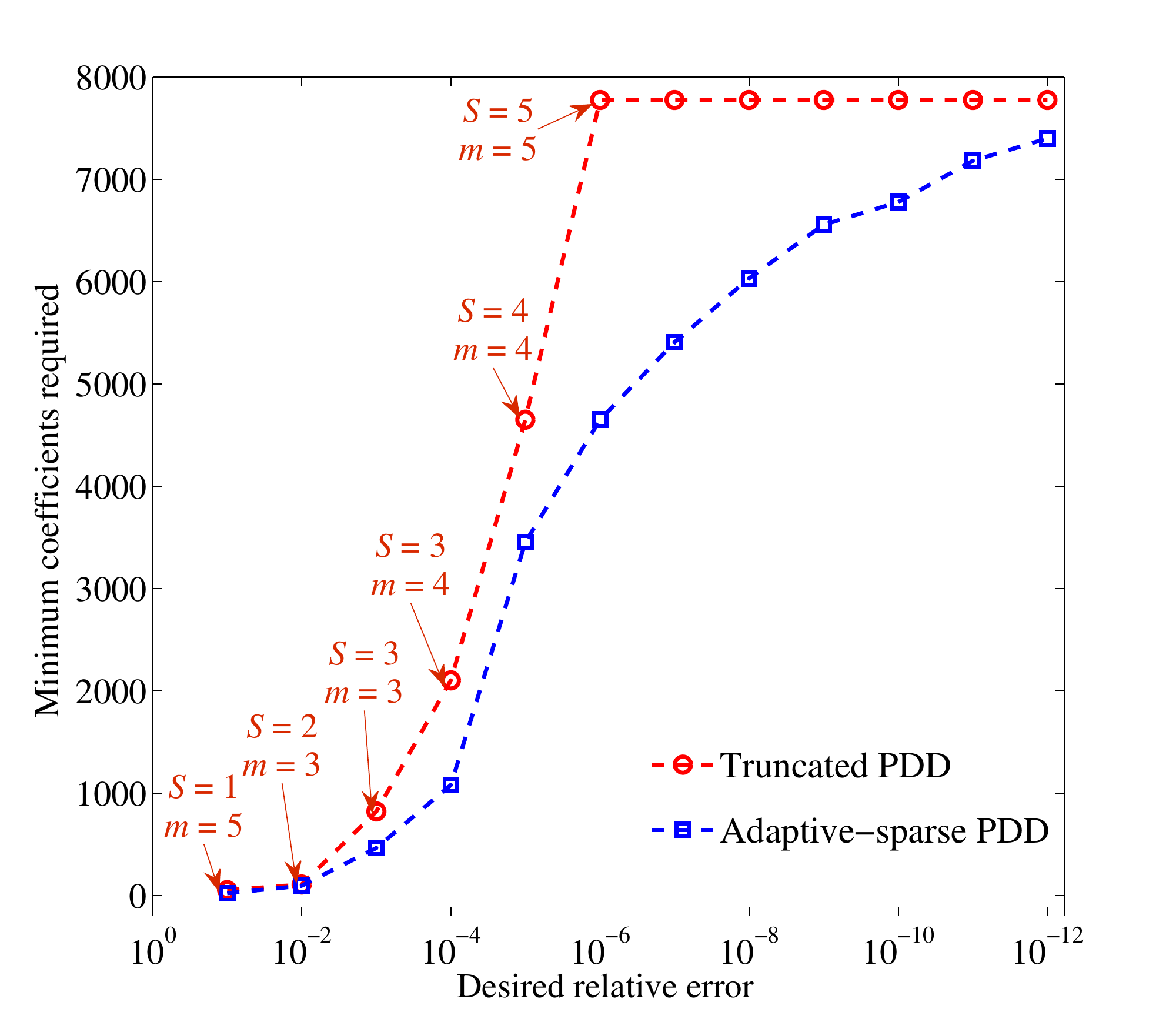}
\par\end{centering}
\caption{\label{ex1-fig2}Minimum number of coefficients required to achieve
a desired relative error in the variance of a mathematical function
by fully adaptive-sparse and truncated PDD methods (Example 1).}
\end{figure}

\subsection{Example 2: Eigenvalues of an Undamped, Spring-Mass System}

Consider a three-degree-of-freedom, undamped, spring-mass system,
shown in Figure \ref{ex2Fig1}, with random mass and random stiffness
matrices
\begin{equation}
\mathbf{M}\left(\mathbf{X}\right)=\left[\begin{array}{ccc}
M_{1}\left(\mathbf{X}\right) & 0 & 0\\
0 & M_{2}\left(\mathbf{X}\right) & 0\\
0 & 0 & M_{3}\left(\mathbf{X}\right)
\end{array}\right]\label{ex2Mass}
\end{equation}
 and
\begin{equation}
\mathbf{K}\left(\mathbf{X}\right)=\left[\begin{array}{ccc}
K_{11}\left(\mathbf{X}\right) & K_{12}\left(\mathbf{X}\right) & K_{13}\left(\mathbf{X}\right)\\
 & K_{22}\left(\mathbf{X}\right) & K_{23}\left(\mathbf{X}\right)\\
(\mathrm{sym.}) &  & K_{33}\left(\mathbf{X}\right)
\end{array}\right],\label{ex2stiffness}
\end{equation}
respectively, where $K_{11}\left(\mathbf{X}\right)=K_{1}\left(\mathbf{X}\right)+K_{4}\left(\mathbf{X}\right)+K_{6}\left(\mathbf{X}\right)$,
$K_{12}\left(\mathbf{X}\right)=-K_{4}\left(\mathbf{X}\right)$, $K_{13}\left(\mathbf{X}\right)=-K_{6}\left(\mathbf{X}\right)$,
$K_{22}\left(\mathbf{X}\right)=K_{4}\left(\mathbf{X}\right)+K_{5}\left(\mathbf{X}\right)+K_{2}\left(\mathbf{X}\right)$,
$K_{23}\left(\mathbf{X}\right)=-K_{5}\left(\mathbf{X}\right)$, and
$K_{33}\left(\mathbf{X}\right)=K_{5}\left(\mathbf{X}\right)+K_{3}\left(\mathbf{X}\right)+K_{6}\left(\mathbf{X}\right)$;
the masses $M_{i}\left(\mathbf{X}\right)=\mu_{i}X_{i}$; $i=1,2,3$
with $\mu_{i}=1.0$ kg; $i=1,2,3$, and spring stiffnesses $K_{i}\left(\mathbf{X}\right)=\mu_{i+3}X_{i+3}$;
$i=1,\cdots,6$ with $\mu_{i+3}=1.0$ N/m; $i=1,\cdots,5$ and $\mu_{9}=3.0$
N/m. The input $\mathbf{X}=\left\{ X_{1},\cdots,X_{9}\right\} ^{T}\in\mathbb{R}^{9}$
is an independent lognormal random vector with mean $\boldsymbol{\mu}_{\mathbf{X}}=\boldsymbol{1}\in\mathbb{R}^{9}$
and covariance matrix $\boldsymbol{\Sigma}_{\mathbf{X}}=\nu^{2}\mathbf{I}\in\mathbb{R}^{9\times9}$
with coefficient of variation $\nu=0.3$.

Three partially adaptive-sparse PDD methods with $S=1$, 2, and 3
were applied to calculate the variances (Equation \eqref{sr14}) of
the three random eigenvalues of the dynamic system. The tolerances
values are as follows: $\epsilon_{1}=\epsilon_{2}=10^{-6}$ and $\epsilon_{3}=0.7$.
Table \ref{ex2table1} presents the variances of eigenvalues from
various partially adaptive-sparse PDD methods calculated according
to Algorithms 1 and 2. The results of both full and reduced ranking
systems are tabulated. Also included in Table \ref{ex2table1} are
the variance calculations from crude MCS. The variances obtained using
the univariate ($S=1$) partially adaptive-sparse PDD approximation
are relatively far from the benchmark results of crude MCS since the
univariate approximation is unable to capture any interactive effects
of the input variables. However, the bivariate ($S=2$) and trivariate
($S=3$) partially adaptive-sparse PDD approximations achieve very
high accuracy in calculating the variances of all three random eigenvalues.
Remarkably, the reduced ranking scheme delivers the same level of
accuracy, at least up to three decimal places shown, of the full ranking
scheme in calculating the variances.

In order to study the efficiency of the reduced ranking scheme vis-a-vis
the full ranking scheme in a trivariate partially adaptive-sparse
PDD approximation, the corresponding total numbers of coefficients
(Equation \eqref{sr17}) required were compared, along with the total
number of coefficients (Equation \eqref{eq:K-tilda-Sm}) required in
a trivariate, fifth-order truncated PDD approximation, in Figure \ref{ex2Fig2}.
The order of the truncated PDD is the largest value of $m_u$ required in the adaptive-sparse PDD approximation. While the partially adaptive-sparse PDD method with either ranking
scheme requires fewer coefficients than does the truncated PDD method,
it is the reduced ranking scheme that is the clear winner in efficiency
with the least number of coefficients. The largest reduction in the
number of coefficients achieved by the reduced ranking system is approximately
sixty-eight percent when calculating the variance of the third eigenvalue. These results are in agreement with Proposition \ref{pro6}.

\begin{figure}[tbph]
\begin{centering}
\centering{}\includegraphics[trim=0.5cm 3.5cm 13cm 1cm, clip=true,scale=0.45, angle=-90]{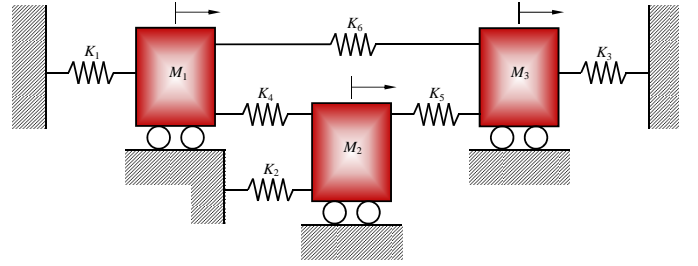}
\par\end{centering}
\caption{\label{ex2Fig1}A three-degree-of-freedom undamped, spring-mass system
(Example 2).}
\end{figure}

\begin{table*}[tbph]
\centering
\begin{threeparttable}
\caption{\label{ex2table1}Variances of three eigenvalues of a three-degree-of-freedom
linear oscillator by three partially adaptive-sparse PDD methods and
crude MCS.}
\begin{centering}
\begin{tabular}{cccccccccccc}
\hline
 &  & \multicolumn{2}{c}{$S=1$} &  & \multicolumn{2}{c}{$S=2$} &  & \multicolumn{2}{c}{$S=3$} &  & MCS\tabularnewline
\hline
$\lambda$  &  & Full  & Reduced  &  & Full  & Reduced  &  & Full  & Reduced  &  & $10^{6}$\tabularnewline
 &  & ranking  & ranking  &  & ranking  & ranking  &  & ranking  & ranking  &  & \tabularnewline
\hline
1  &  & $0.057$  & $0.057$  &  & $0.060$  & $0.060$  &  & $0.060$  & $0.060$  &  & $0.060$\tabularnewline
2  &  & $1.152$  & $1.152$  &  & $1.204$  & $1.204$  &  & $1.215$  & $1.215$  &  & $1.219$\tabularnewline
3  &  & $7.289$  & $7.289$  &  & $7.576$  & $7.576$  &  & $7.585$  & $7.585$  &  & $7.585$\tabularnewline
\hline
\end{tabular}
\par\end{centering}
\end{threeparttable}
\end{table*}

\begin{figure}[tbph]
\begin{centering}
\includegraphics[trim=2cm 8.5cm 1cm 0cm, clip=true,scale=0.53]{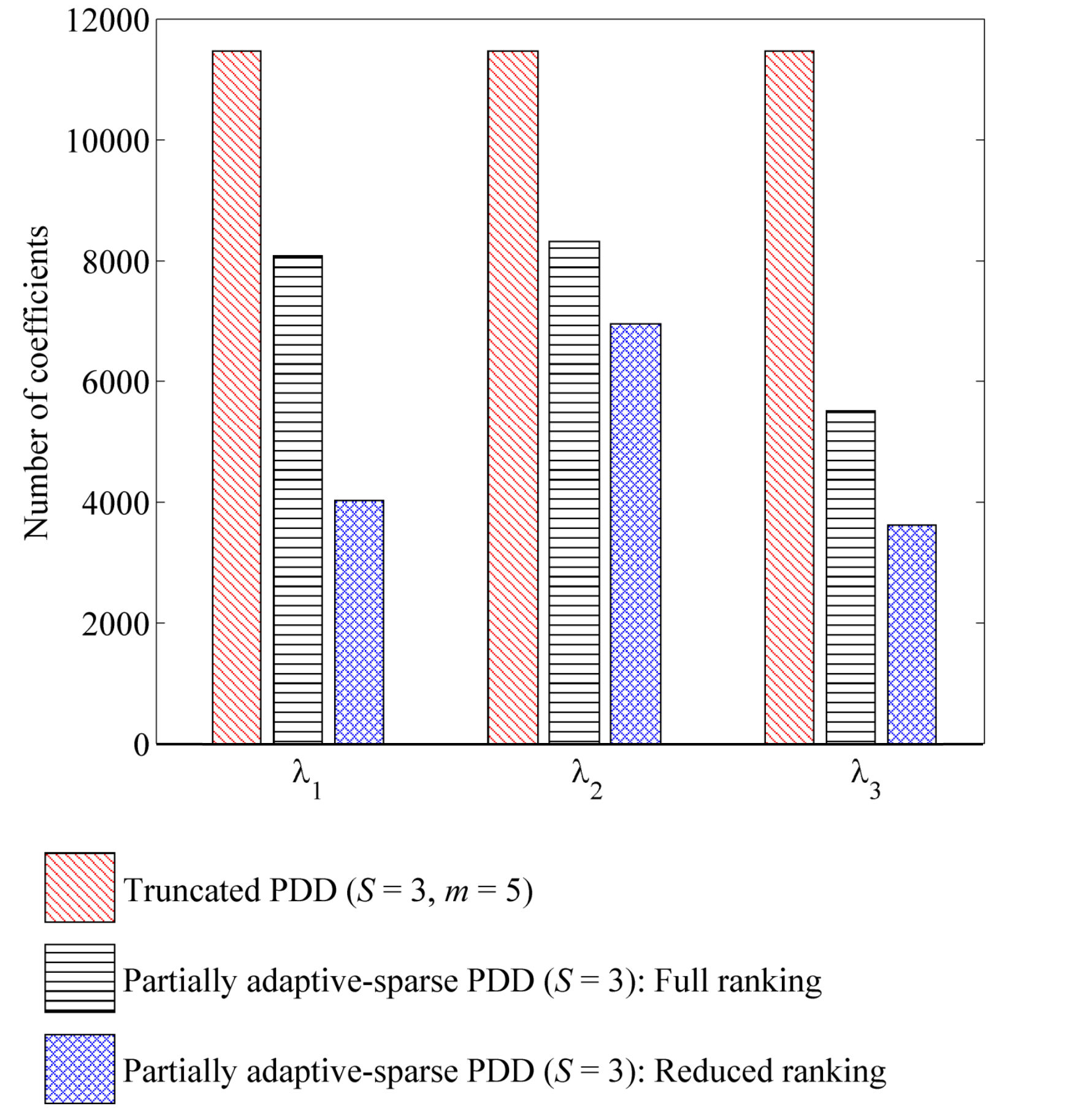}
\par\end{centering}
\caption{\label{ex2Fig2}Number of coefficients required for calculating the
variance of a three-degree-of-freedom linear oscillator by trivariate
partially adaptive-sparse PDD approximations using full and reduced
ranking schemes.}
\end{figure}

\subsection{Example 3: Modal Analysis of a Functionally Graded Cantilever Plate}

The third example involves free vibration analysis of a $2\mathrm{\: m}\times1\mathrm{\: m\times10}\:\mathrm{mm}$
cantilever plate, shown in Figure \ref{ex3fig1}(a), made of a functionally
graded material (FGM)%
\footnote{Functionally graded materials are two- or multi-phase particulate
composites in which material composition and microstructure vary spatially
in the macroscopic length scale to meet a desired functional performance. %
}, where silicon carbide (SiC) particles varying along the horizontal
coordinate $\xi$ are randomly dispersed in an aluminum (Al) matrix
\cite{yadav13}. The result is a random inhomogeneous plate, where the effective
elastic modulus $E(\xi)$, effective Poisson's ratio $\nu(\xi)$,
and effective mass density $\rho(\xi)$ are random fields. They depend
on two principal sources of uncertainties: (1) randomness in the volume
fraction of SiC particles $\phi_{\mathrm{SiC}}(\xi)$, which varies
only along $\xi$, and (2) randomness in constituent material properties,
comprising elastic moduli $E_{\mathrm{SiC}}$ and $E_{\mathrm{Al}}$,
Poisson's ratios $\nu_{\mathrm{SiC}}$ and $\nu_{\mathrm{Al}}$, and
mass densities $\rho_{\mathrm{SiC}}$ and $\rho_{\mathrm{Al}}$ of
SiC and Al material phases, respectively. The particle volume fraction
$\phi_{\mathrm{SiC}}(\xi)$ is a one-dimensional, inhomogeneous, Beta
random field with mean $\mu_{\mathrm{SiC}}(\xi)=1-\xi/L$, standard
deviation $\sigma_{\mathrm{SiC}}(\xi)=(\xi/L)(1-\xi/L)$, where $L$
is the length of the plate. Assuming an appropriately bounded covariance
function of $\phi_{\mathrm{SiC}}(\xi)$, the standardized volume fraction,
$\tilde{\phi}_{\mathrm{SiC}}(\xi):=[\phi_{\mathrm{SiC}}(\xi)-\mu_{\mathrm{SiC}}(\xi)]/\sigma_{\mathrm{SiC}}(\xi)$,
was mapped to a \emph{zero}-mean, homogeneous, Gaussian image field $\alpha(\xi)$
with an exponential covariance function $\Gamma_{\alpha}(t):=\mathbb{E}[\alpha(\xi)\alpha(\xi+t)]=\exp(-\left|t\right|/0.125L)$
via $\tilde{\phi}_{\mathrm{SiC}}(\xi)=F_{\mathrm{SiC}}^{-1}\left[\Phi(\alpha(\xi))\right]$,
where $\Phi$ is the distribution function of a standard Gaussian
random variable and $F_{\mathrm{SiC}}$ is the marginal distribution
function of $\tilde{\phi}_{\mathrm{SiC}}(\xi)$. The Karhunen-Lo\`{e}ve
approximation \cite{davenport58} was employed to discretize $\alpha(\xi)$
and hence $\phi_{\mathrm{SiC}}(\xi)$ into 28 standard Gaussian random
variables. In addition, the constituent material properties, $E_{\mathrm{SiC}}$,
$E_{\mathrm{Al}}$, $\nu_{\mathrm{SiC}}$, $\nu_{\mathrm{Al}}$, $\rho_{\mathrm{SiC}}$, and
$\rho_{\mathrm{Al}}$, were modeled as independent lognormal random
variables with their means and coefficients of variation described
in Table \ref{tab:3}. Therefore, a total of 34 random variables are
involved in this example. Employing a rule of mixture, $E(\xi)\cong E_{\mathrm{SiC}}\phi_{\mathrm{SiC}}(\xi)+E_{\mathrm{Al}}[1-\phi_{\mathrm{SiC}}(\xi)]$,
$\nu(\xi)\cong\nu_{\mathrm{SiC}}\phi_{\mathrm{SiC}}(\xi)+\nu_{\mathrm{Al}}[1-\phi_{\mathrm{SiC}}(\xi)]$,
and $\rho(\xi)\cong\rho_{\mathrm{SiC}}\phi_{\mathrm{SiC}}(\xi)+\rho_{\mathrm{Al}}[1-\phi_{\mathrm{SiC}}(\xi)]$.
Using these spatially-variant effective properties, a $20\times40$
mesh consisting of 800 eight-noded, second-order shell elements, shown
in Figure \ref{ex3fig1}(b), was constructed for FEA, to determine the natural frequencies of the FGM plate. No damping
was included. A Lanczos algorithm \cite{cullum02} was employed for
calculating the eigenvalues.

\begin{figure}[tbph]
\begin{centering}
\includegraphics[scale=0.8]{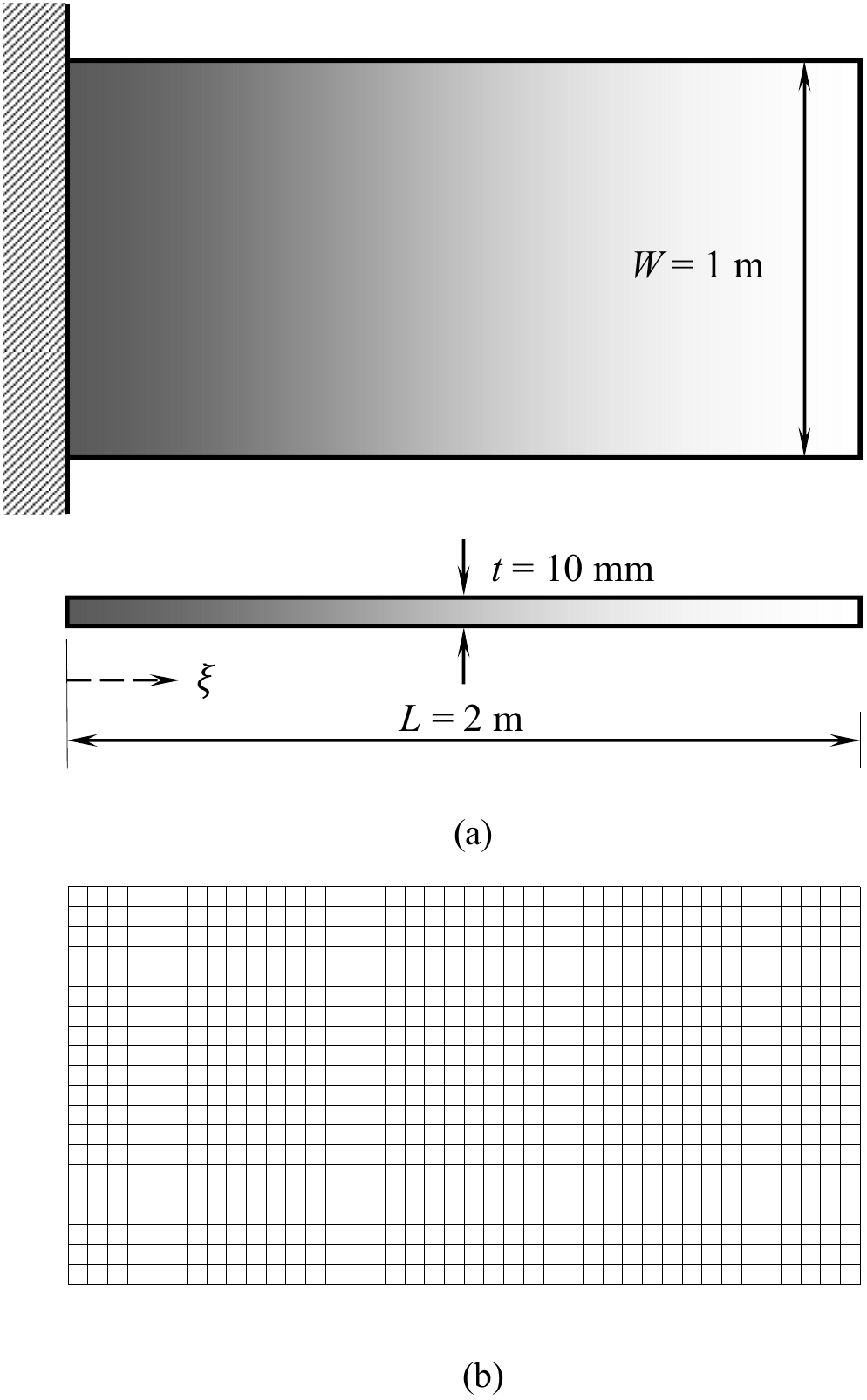}
\par\end{centering}
\caption{\label{ex3fig1}An FGM cantilever plate: (a) geometry; (b) a 20$\times$40
FEA mesh.}
\end{figure}

\begin{table}[tbph]
\centering
\begin{threeparttable}
 \caption{\label{tab:3}Statistical material properties of constituents in
SiC-Al FGM.}

\begin{centering}
\begin{tabular}{lcc}
\hline
Material properties{\footnotesize $^{(\mathrm{1})}$}\hspace{10 mm}
& Mean  \hspace{15 mm}&\hspace{5 mm} COV{\footnotesize $^{(\mathrm{2})}$}, \% \hspace{5 mm}\tabularnewline
\hline
$E_{\mathrm{SiC}}$, GPa  & 419.2  & 15\tabularnewline
$\nu_{\mathrm{SiC}}$  & 0.19  & 5\tabularnewline
$\rho_{\mathrm{SiC}}$, kg/$\mathrm{m}^{3}$  & 3210  & 15\tabularnewline
$E_{\mathrm{Al}}$, GPa  & 69.7  & 15\tabularnewline
$\nu_{\mathrm{Al}}$  & 0.34  & 5\tabularnewline
$\rho_{\mathrm{Al}}$, kg/$\mathrm{m}^{3}$  & 2520  & 15\tabularnewline
\hline
\end{tabular}
\begin{tablenotes}
\footnotesize
\item [(1)] $E_{\mathrm{SiC}}$ = elastic modulus of SiC, $\nu_{\mathrm{SiC}}$ = Poisson's ratio of SiC,
\item \hspace{1 mm}$\rho_{\mathrm{SiC}}$ = mass density of SiC, $E_{\mathrm{Al}}$ = elastic modulus of Al,
\item \hspace{1 mm}$\nu_{\mathrm{Al}}$ = Poisson's ratio of Al, $\rho_{\mathrm{Al}}$ = mass density of Al.
\item [(2)] Coefficient of variation.
\end{tablenotes}
\par\end{centering}
\end{threeparttable}
\end{table}

The probability distributions of the first six natural frequencies
of the functionally graded material plate were evaluated using four
different PDD methods: (1) the bivariate partially adaptive-sparse
PDD method with full-grid dimension-reduction integration; (2) the
bivariate partially adaptive-sparse PDD method with sparse-grid dimension-reduction
integration with extended FSI rule; (3) the univariate, fifth-order
PDD method; and (4) the bivariate, fifth-order PDD method; and the crude
MCS. Again, the order of the truncated PDD was selected based on the largest value of $m_u$ required in the adaptive-sparse PDD methods. The tolerances used
for adaptive and ranking algorithms are $\epsilon_{1}=\epsilon_{2}=10^{-6}$
and $\epsilon_{3}=0.9$. Figure \ref{ex3fig2} presents the marginal
probability distributions $F_{i}(\omega_{i}):=P[\Omega_{i}\le\omega_{i}]$
of the first six natural frequencies $\Omega_{i}$, $i=1,\cdots,6$,
where all the PDD solutions were obtained from the embedded MCS. The
plots are made over a semi-logarithmic scale to delineate the distributions
in the tail regions. For all six frequencies, the probability distributions
obtained from a bivariate partially adaptive-sparse PDD method, whether
using either full-grid or sparse-grid, and the bivariate fifth-order
PDD method are much closer to the crude Monte Carlo results compared
with those obtained from the univariate, fifth-order PDD method. While all PDD approximations
require fewer function evaluations than the crude MCS, both variants of the partially
adaptive-sparse PDD approximations remit exceptionally high efficiency
by an average factor of six when compared with the bivariate, fifth-order
PDD approximation. However, the advantage of the sparse-grid integration
over the full-grid integration employed in the adaptive-sparse approximation
is modest in terms of computational efficiency. This is explained as follows.

\begin{figure*}[tbph]
\includegraphics[scale=0.8]{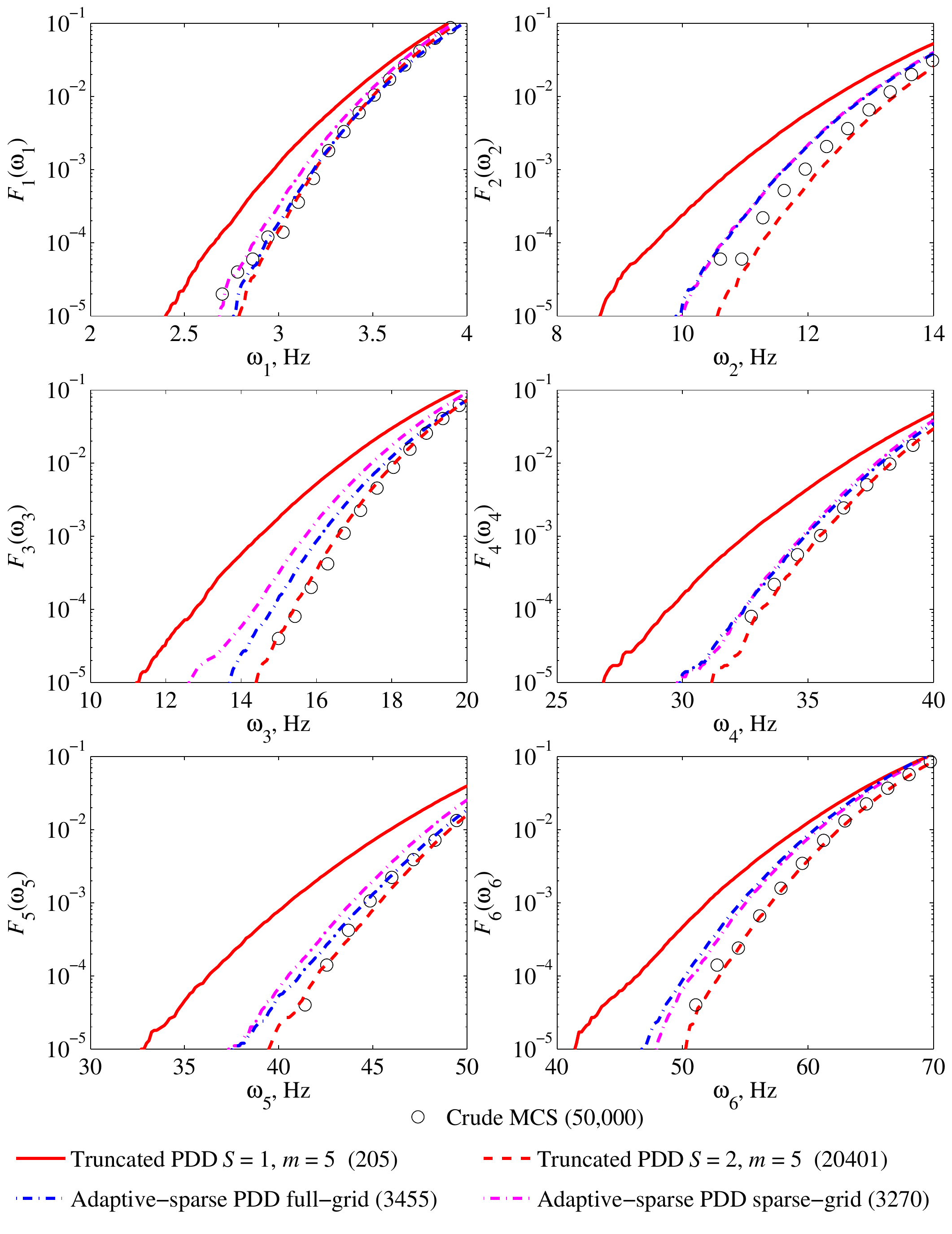}
\caption{\label{ex3fig2}Marginal probability distributions of the first six
natural frequencies of an FGM plate by various PDD approximations
and crude MCS.}
\end{figure*}

The efficient reduced ranking algorithm was employed
in this example. When the bivariate component functions were ranked
for $m_{u}=1$, the coefficient calculation for both full-grid and
sparse-grid involved function evaluation at the point $\left(0,0\right)$
as shown for $l=1$ in Figure \ref{SparseFig1}. The function evaluations at this
point return only the functions already evaluated at the point $\left(\bold{c}\right)$,
i.e., response at mean $y\left(\bold{c}\right)$, thus the bivariate
component functions could not be ranked for $m_{u}=1$. When the polynomial
order was incremented to $m_{u}=2$, the full-grid for $l=2$
comprises of four non-\emph{zero} integration points, resulting in non-trivial
bivariate function evaluations at those points. However, the sparse-grid
consists of four new points lying on the axes, failing to capture
the interaction effect of two variables. This results in bivariate
function evaluations that are not useful in creating a ranking. Thus,
for $m_{u}=2$, full-grid involves ranking all the $28\times27/2=378$
bivariate component functions, with $378\times4=1512$ new function
evaluations, while the sparse-grid was still lacking any ranking.
Moving to $m_{u}=3$, full-grid can afford to exploit the efficient
reduced-ranking by truncating the ranking from $m_{u}=2$ and calculating
coefficients only for fewer than $378$ component functions. However, the sparse-grid
is forced to evaluate all $378$ component functions for $m_{u}=3$,
resulting in $378\times4=1512$ function evaluations at four new integration
points, depriving this efficient technique of any initial advantage.
The modest advantage in computational efficiency that the sparse-grid
eventually achieves was obtained only after ranking at $m_{u}=4$
and onwards.

\begin{figure*}[tbph]
\includegraphics[scale=0.85]{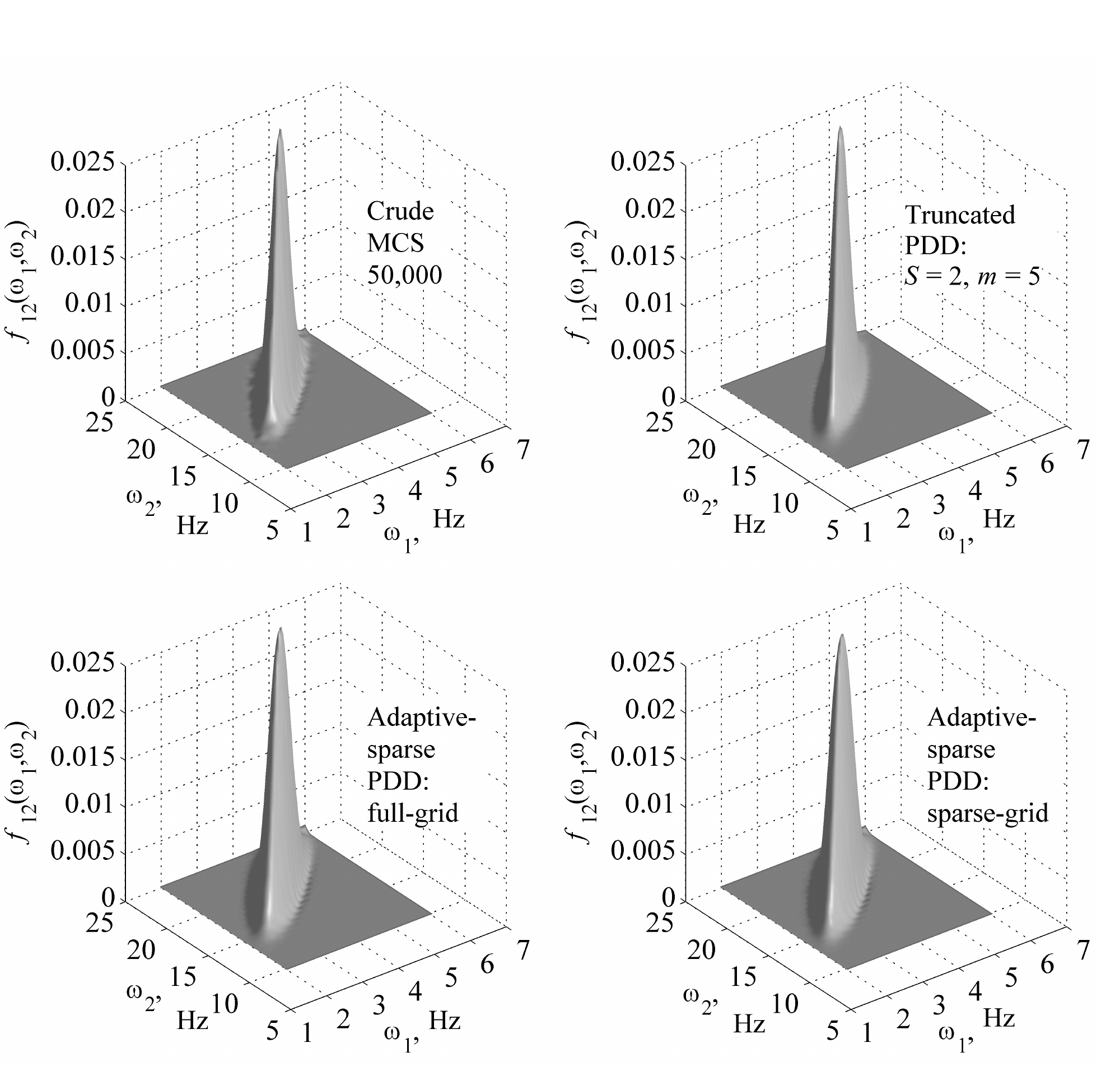}
\caption{\label{ex3fig3}Joint probability density function of the first and
second natural frequencies of the FGM plate by various PDD approximations
and crude MCS.}
\end{figure*}

Figure \ref{ex3fig3} displays the joint probability density function
$f_{12}(\omega_{1},\omega_{2})$ of the first two natural frequencies
$\Omega_{1}$ and $\Omega_{2}$ obtained by the two variants of the
bivariate partially adaptive-sparse PDD method, the bivariate, fifth-order
PDD method, and crude MCS. Although visually comparing these three-dimensional
plots is not simple, the joint distributions from all PDD approximations
and the crude Monte Carlo method seem to match reasonably well. The
contours of these three-dimensional plots were studied at two notably
different levels: $f_{12}=0.005$ (high level) and $f_{12}=0.0005$
(low level), as depicted in Figures \ref{ex2fig4}(a) and \ref{ex2fig4}(b),
respectively. For both levels examined, a good agreement exists among
the contours from all four distributions. These results are consistent
with the marginal distributions of natural frequencies discussed in
the preceding paragraph.

\begin{figure}[tbph]
\centering{}\includegraphics[scale=0.57]{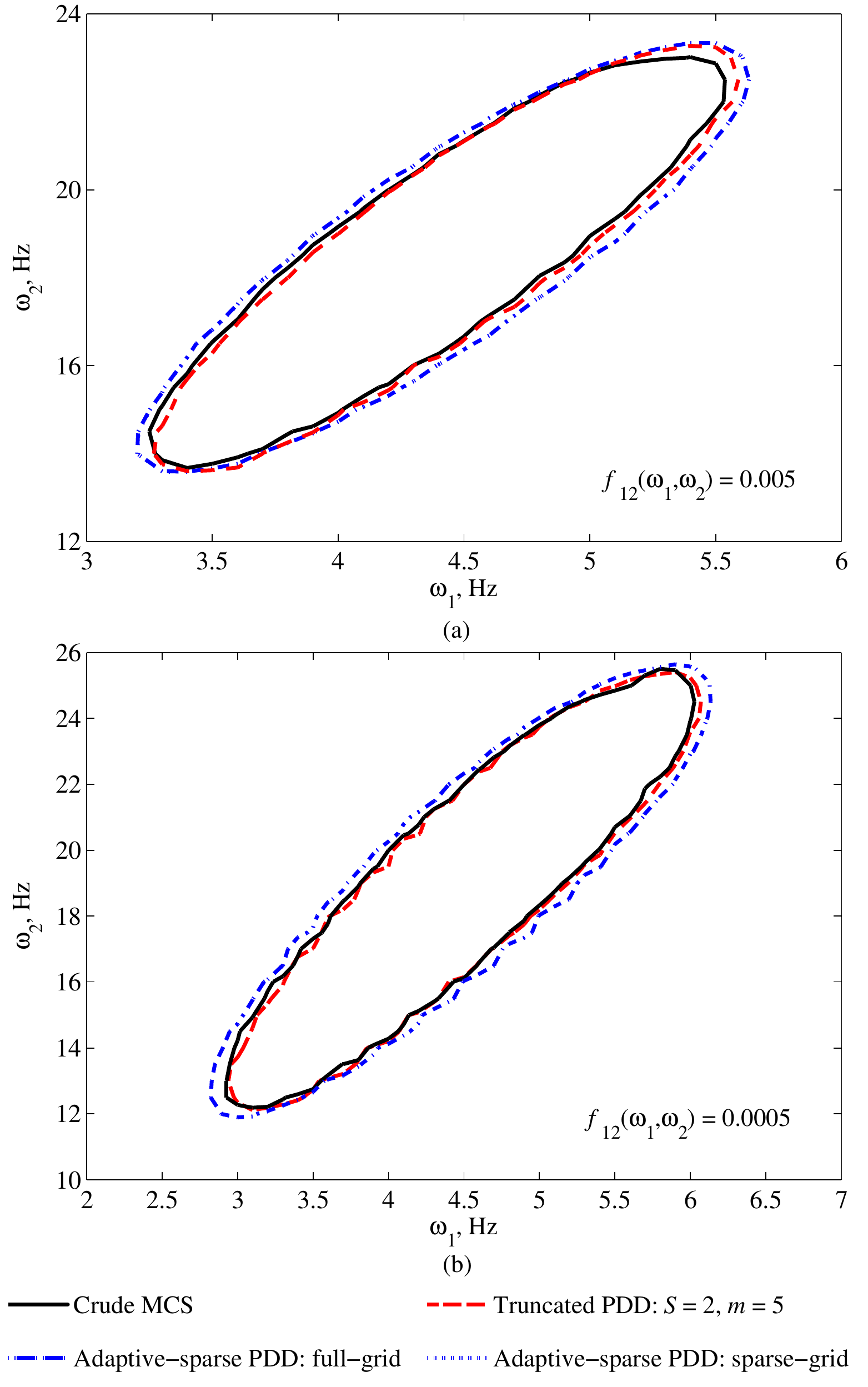}
\caption{\label{ex2fig4}Contours of the joint density function of the first
and second natural frequencies of the FGM plate by various PDD approximations
and crude MCS: (a) $f_{12}=0.005$; (b) $f_{12}=0.0005$.}
\end{figure}

\section{Application: A Disk Brake System}

This section demonstrates the capabilities of the proposed partially
adaptive-sparse PDD method in solving a large-scale practical engineering
problem. The application comprises of determining instabilities in
a disk brake system in terms of statistical analysis of complex frequencies
and corresponding mode shapes. The dynamic instabilities in a braking
system, emanating from complex frequencies, give rise to the highly
undesired phenomenon of brake squeal. When a braking system is subjected
to random input parameters, it is imperative to perform a random
brake-squeal analysis in order to identify, quantify, and minimize
the random dynamic instabilities.

\subsection{Brake-Squeal Analysis}

A disk brake system, illustrated in Figure \ref{brakeFEA}(a), slows
motion of the wheel by pushing brake pads against a rotor with a set
of calipers \cite{wikibrake}. The brake pads mounted on a brake caliper is forced mechanically,
hydraulically, pneumatically, or electromagnetically against both
sides of the rotor. Friction causes the rotor and attached wheel to
slow or stop. Figure \ref{brakeFEA}(b) presents a simplified FEA
model of a disk brake system commonly used in domestic passenger vehicles.
The system consists of a rotor of diameter 288 mm and thickness 20
mm. Two pads are positioned on both sides of the rotor. Assembled
behind the pads are back plates and insulators. The FEA mesh of the
model consists of 26,125 elements and 111,129 active degrees of freedom
and was generated using C3D6 and C3D8I elements in Abaqus computer
software (Version 6.12) \cite{abaqus11}. The rotor is made of cast
iron and the back plates and insulators are made of steel. The two
brake pads are made of organic frictional material, which is modeled
as an orthotropic elastic material. The mass densities and Young's
moduli of the rotor, back-plates, insulators and pads along with the
shear moduli of the pads are modeled as random variables with uniform
distribution. Along with the random material properties, the brake
pressure, the radial velocity of the rotor, and the coefficient of
friction between the rotor and pads are modeled as uniform random
variables, constituting a total of $16$ random variables in this
problem. The statistical properties of all random variables are listed
in Table \ref{brakeRVs}. Apart from the random material properties,
the deterministic Poisson's ratio of rotor, back-plates and insulators
are 0.24, 0.28, and 0.29, respectively. The three Poisson's ratios
of orthotropic material of pads are $\nu_{12}=0.06$, $\nu_{23}=0.41$,
and $\nu_{31}=0.15$.

\begin{figure}[tbph]
\centering{}\includegraphics[scale=0.58]{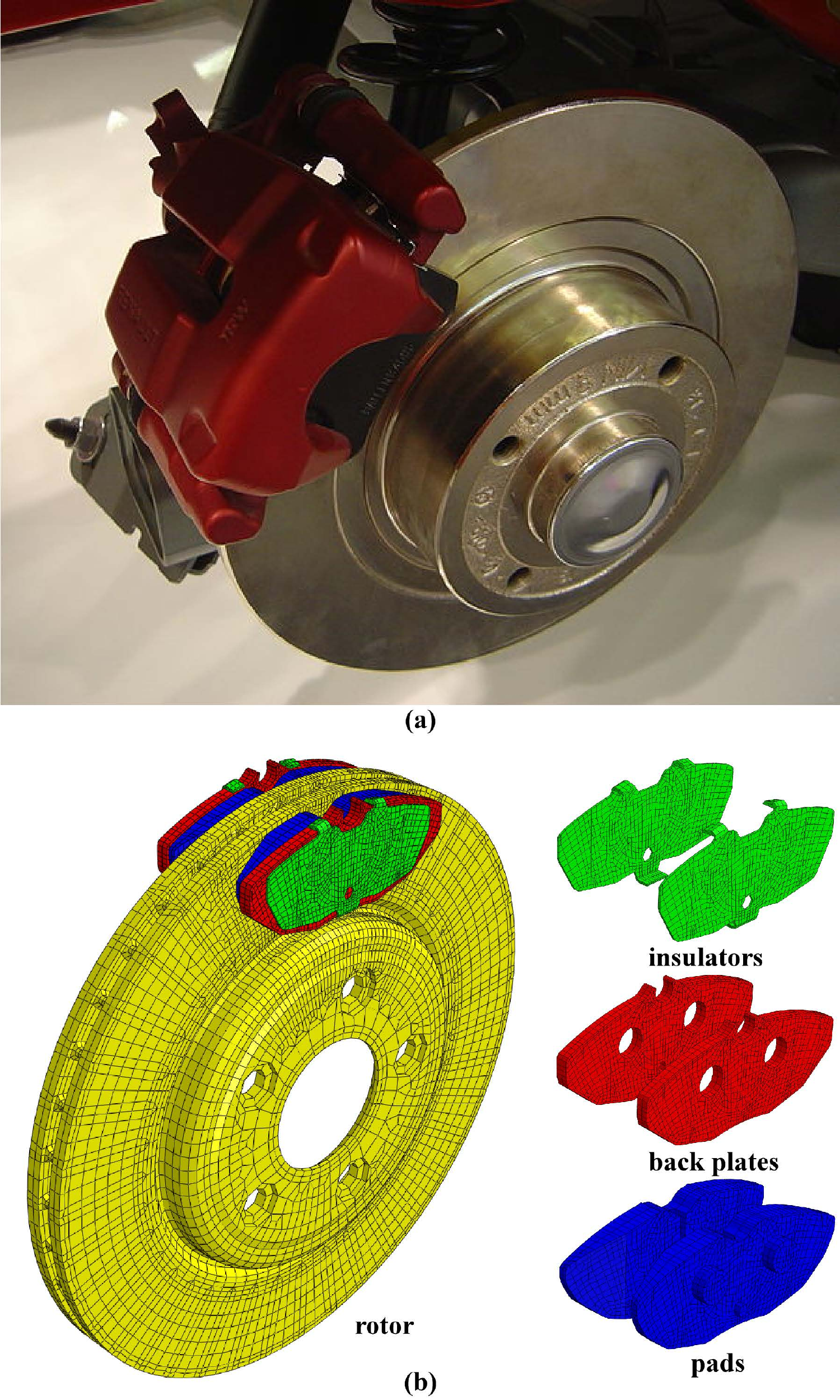}
 \caption{\label{brakeFEA}A disk brake system with various mechanical components:
(a) close-up on a passenger vehicle; (b) a simplified FEA model }
\end{figure}

\begin{table}[tbph]
\caption{\label{brakeRVs}Random input variables in disk-brake system with
the minimum $\left(a_{i}\right)$ and maximum $\left(b_{i}\right)$
values of their uniform distributions.}

\begin{tabular}{lcc}
\hline
Random variables{\footnotesize $^{(\mathrm{1})}$}  & \hspace{0.25in}$a_{i}$  & \hspace{0.25in}$b_{i}$\tabularnewline
\hline
$\rho_{\mathrm{rotor}}$, $\mathrm{kg/mm^{3}}$  & \hspace{0.25in}$5.329\times10^{-6}$  & \hspace{0.25in}$9.071\times10^{-6}$\tabularnewline
$\rho_{\mathrm{back\, plate}}$, $\mathrm{kg/mm^{3}}$  & \hspace{0.25in}$5.788\times10^{-6}$  & \hspace{0.25in}$9.851\times10^{-6}$\tabularnewline
$\rho_{\mathrm{insulator}}$, $\mathrm{kg/mm^{3}}$  & \hspace{0.25in}$5.788\times10^{-6}$  & \hspace{0.25in}$9.851\times10^{-6}$\tabularnewline
$\rho_{\mathrm{pad}}$, $\mathrm{kg/mm^{3}}$  & \hspace{0.25in}$1.858\times10^{-6}$  & \hspace{0.25in}$3.162\times10^{-6}$\tabularnewline
$E_{\mathrm{rotor}}$, GPa  & \hspace{0.25in}92.52  & \hspace{0.25in}157.5\tabularnewline
$E_{\mathrm{back\, plate}}$, GPa  & \hspace{0.25in}153.2  & \hspace{0.25in}260.8\tabularnewline
$E_{\mathrm{insulator}}$, GPa  & \hspace{0.25in}153.2  & \hspace{0.25in}260.8\tabularnewline
$E_{1,\mathrm{pad}}$, GPa  & \hspace{0.25in}4.068  & \hspace{0.25in}6.924\tabularnewline
$E_{2,\mathrm{pad}}$, GPa  & \hspace{0.25in}4.068  & \hspace{0.25in}6.924\tabularnewline
$E_{3,\mathrm{pad}}$, GPa  & \hspace{0.25in}1.468  & \hspace{0.25in}2.498\tabularnewline
$G_{12,\mathrm{pad}}$, GPa  & \hspace{0.25in}1.917  & \hspace{0.25in}3.263\tabularnewline
$G_{13,\mathrm{pad}}$, GPa  & \hspace{0.25in}$0.873$  & \hspace{0.25in}$1.486$\tabularnewline
$G_{23,\mathrm{pad}}$, GPa  & \hspace{0.25in}$0.873$  & \hspace{0.25in}$1.486$\tabularnewline
$P$, $\mathrm{kg/mm^{2}}$  & \hspace{0.25in}370.1  & \hspace{0.25in}629.9\tabularnewline
$\omega$, rad/s  & \hspace{0.25in}3.701  & \hspace{0.25in}6.299\tabularnewline
$\mu$  & \hspace{0.25in}0.50  & \hspace{0.25in}0.70\tabularnewline
\hline
\end{tabular}

{\footnotesize (1) $\rho_{\mathrm{rotor}}$, $\rho_{\mathrm{back\, plate}}$,
$\rho_{\mathrm{insulator}}$, $\rho_{\mathrm{pad}}$: mass densities
of corresponding materials, }{\footnotesize \par}

{\footnotesize \hspace{1mm}$E_{\mathrm{rotor}}$, $E_{\mathrm{back\, plate}}$,
$E_{\mathrm{insulator}}$: elastic modulus of corresponding materials, }{\footnotesize \par}

{\footnotesize \hspace{1mm}$E_{1,\mathrm{pad}}$, $E_{2,\mathrm{pad}}$,
$E_{3,\mathrm{pad}}$: elastic modulus associated with the normal
directions of pad material, }{\footnotesize \par}

{\footnotesize \hspace{1mm}$G_{12,\mathrm{pad}}$, $G_{13,\mathrm{pad}}$,
$G_{23,\mathrm{pad}}$: shear modulus associated with the principal
directions of pad material, }{\footnotesize \par}

{\footnotesize \hspace{1mm}$P$: brake pressure, $\omega$: radial
velocity, $\mu$: friction coefficient.}
\end{table}

\begin{figure*}[tbph]
\centering{}\includegraphics[scale=0.8]{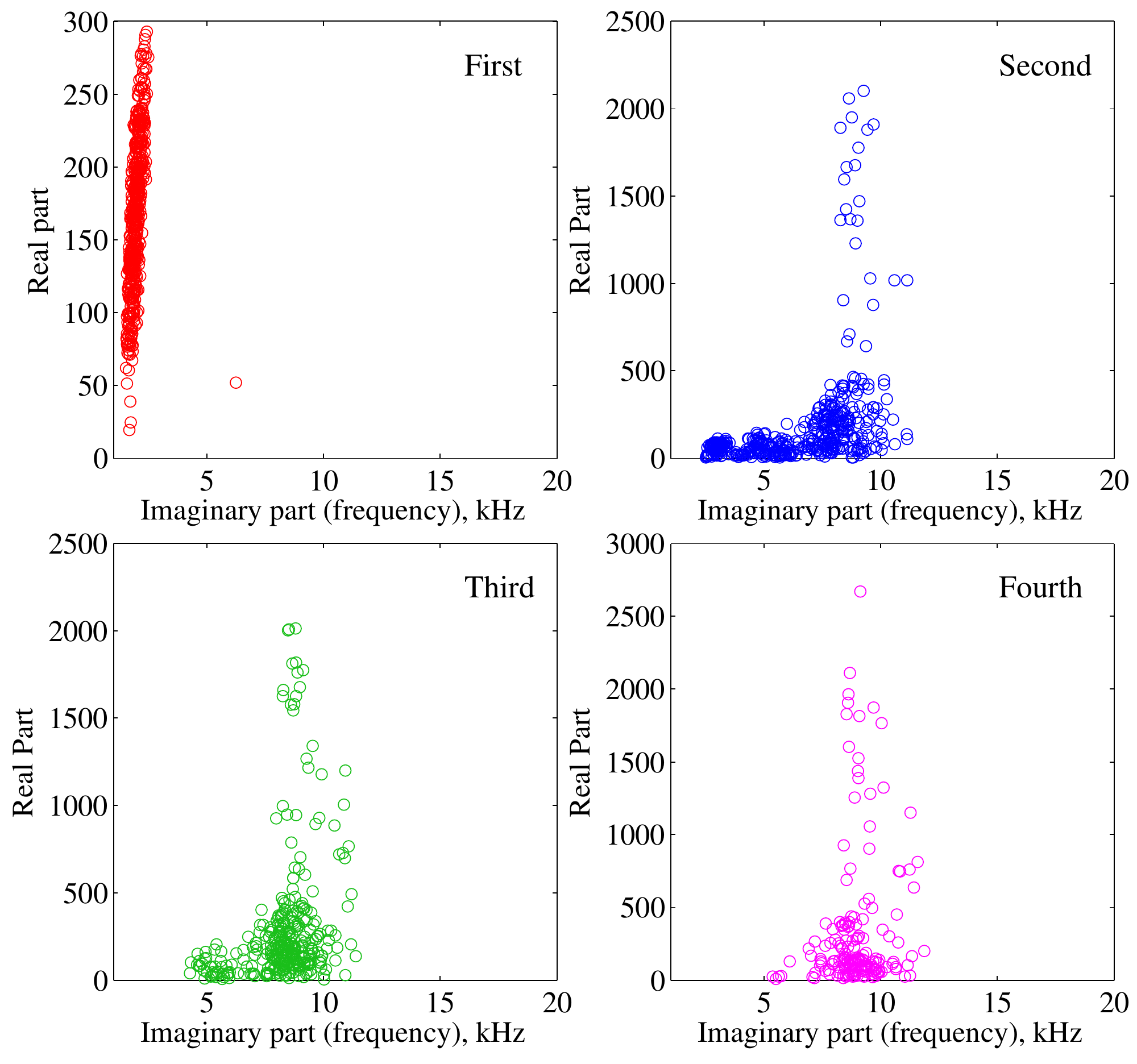}
\caption{\label{brakeFreq}Complex eigenvalues of a disk brake system for
first four unstable modes.}
\end{figure*}

\subsection{Results}

The dynamic analysis was performed in four steps. In the first step,
contact was established between the rotor and the pad by applying
brake pressure to the external surfaces of the insulators. Braking
at low velocity was simulated in the second step by imposing a rotational
velocity on the rotor, accompanied with an introduction of a non-\emph{zero}
friction coefficient between rotor and pad. In the third step, natural
frequencies up to 20 kHz were extracted by the eigenvalue extraction
procedure in the steady-state condition using the automatic multilevel
substructuring method with subspace projection in Abaqus. Finally,
in the fourth step a complex eigenvalue analysis was performed up
to the first 55 modes.

The bivariate partially adaptive-sparse PDD method with tolerances
$\epsilon_{1}=\epsilon_{2}=10^{-6}$, $\epsilon_{3}=0.9$ was applied
to determine the probabilistic characteristics of the dynamic instabilities
caused by the first two unstable modes of the disk brake system. Since
all input random variables are uniformly distributed, classical Legendre
orthonormal polynomials were used as basis functions. The PDD coefficients
were calculated using the quasi MCS with 500 samples generated from
a 16-dimensional low-discrepancy Sobol sequence. The sample size, although selected arbitrarily, is adequate, as there exist no significant changes to the coefficients, at least, for this problem. Figure \ref{brakeFreq}
displays real and imaginary parts of the eigenvalues of the first four unstable modes obtained in each quasi Monte
Carlo sample. These unstable modes, conveyed by complex frequencies
with positive real parts, reflect the dynamic instability caused in
the brake system. Each occurrence of the unstable frequency may cause
the brake to squeal.

Equations \eqref{sr12} and \eqref{sr14} were employed to calculate the
second-moment statistics of each nodal displacement component of an
eigenvector describing the associated mode shape of the disk brake
system. Based on these statistics, the $\mathcal{L}_{2}$-norms, that
is, the square root of sum of squares, of the mean and variance of
a nodal displacement were calculated. Figures \ref{brakeMoments}(a)
and \ref{brakeMoments}(b) present contour plots of the $\mathcal{L}_{2}$-norms
of the means and variances, respectively, of the first two unstable
mode shapes of the disk brake system. Similar results can be generated
for other mode shapes, stable or unstable, if desired.

For a disk brake system with complex frequencies, the $i$th effective
damping ratio is defined as $-2\mathrm{Re}\left[\lambda_{u}^{\left(i\right)}\left(\mathbf{X}\right)\right]/\mathrm{Im}|\lambda_{u}^{\left(i\right)}\left(\mathbf{X}\right)|$,
where $\mathrm{Re}\left[\lambda_{u}^{\left(i\right)}\left(\mathbf{X}\right)\right]$
and $\mathrm{Im}|\lambda_{u}^{\left(i\right)}\left(\mathbf{X}\right)|$
are the real part and the imaginary part, respectively, of the $i$th
unstable frequency $\lambda_{u}^{\left(i\right)}\left(\mathbf{X}\right)$.
The magnitude of the damping ratio represents the harshness of brake
squeal. The larger the magnitude of the damping ratio, the higher
the propensity for brake squeal. Figure \ref{brakePDF} illustrates
the marginal probability density functions of the effective damping
ratios corresponding to the first two unstable modes. These probability
densities provide a measure of the effect of random input parameters
on the dynamic instabilities caused in the disk brake system.

It is worth mentioning that a similar brake-squeal analysis with only
five input random variables was performed using a univariate RDD method
\cite{rahman07}. However, verification or improvement of the univariate solution
was not possible due to inherent limitations of the method used. The
adaptive-sparse PDD approximations developed in this work have overcome
this quandary even for a significantly more input variables.

\begin{figure*}[tbph]
\includegraphics[scale=0.85]{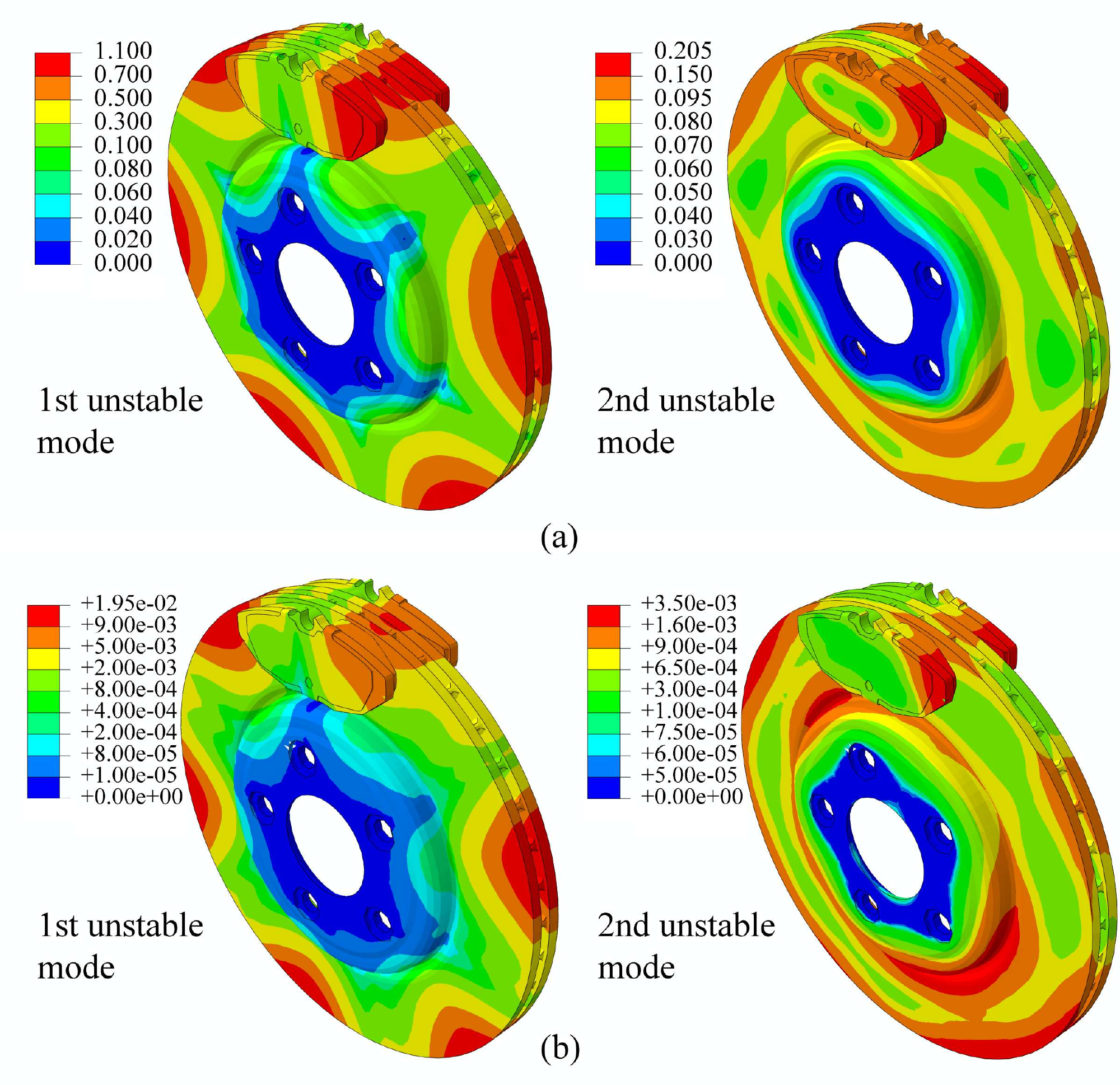}
\caption{\label{brakeMoments}Contour plots of the $\mathcal{L}_{2}$-norm
of the first two unstable mode shapes of a disk brake system by the
bivariate partially adaptive-sparse PDD method: (a) mean; (b) variance.}
\end{figure*}

\begin{figure}[tbph]
\centering{}\includegraphics[scale=0.8]{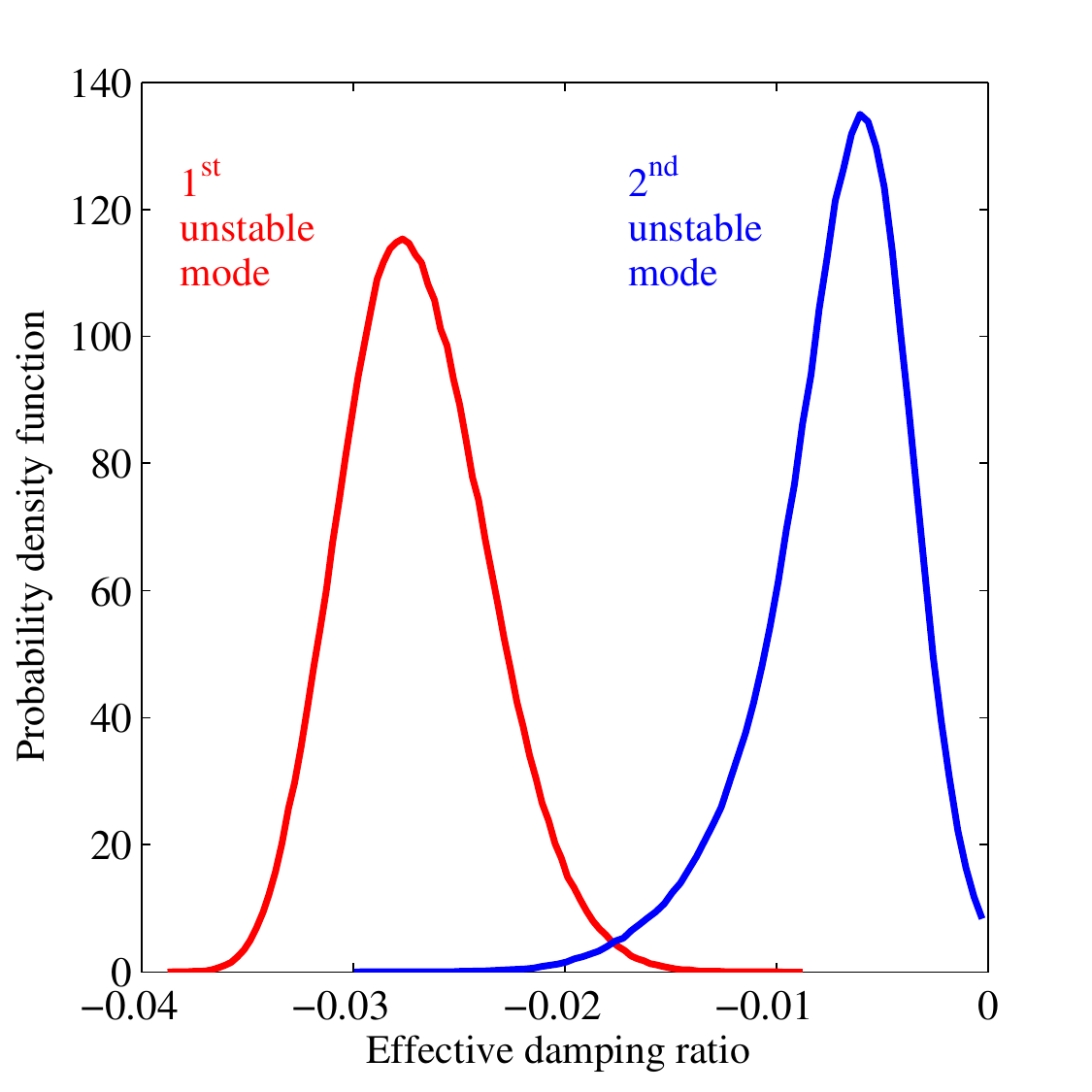}
\caption{\label{brakePDF}Marginal probability density functions of the effective
damping ratios of first two unstable modes of a disk brake system
by the bivariate partially adaptive-sparse PDD method.}
\end{figure}

\section{Conclusions}
Two new adaptive-sparse PDD methods, the fully adaptive-sparse PDD method and a partially adaptive-sparse PDD method, were developed for uncertainty
quantification of high-dimensional complex systems commonly encountered
in applied sciences and engineering. The methods are based on global
sensitivity analysis for defining the pruning criteria to retain important PDD component functions, and a full-
or sparse-grid dimension-reduction integration or quasi MCS
for estimating the PDD expansion coefficients. In the fully
adaptive-sparse PDD approximation, PDD component functions of an arbitrary
number of input variables are retained by truncating the degree of
interaction among input variables and the order of orthogonal polynomials
according to specified tolerance criteria. In a partially adaptive-sparse
PDD approximation, PDD component functions with a specified degree
of interaction are retained by truncating the order of orthogonal
polynomials, fulfilling relaxed tolerance criteria. The former approximation
is comprehensive and rigorous, leading to the second-moment statistics
of a stochastic response that converges to the exact solution when
the tolerances vanish. The latter approximation, obtained through
regulated adaptivity and sparsity, is more economical than the former
approximation and is, therefore, expected to solve practical problems
with numerous variables. A unified computational algorithm was created
for solving a general stochastic problem by the new PDD methods. Two
distinct ranking schemes $-$ full ranking and reduced ranking $-$
were also developed for grading PDD component functions in the unified
algorithm. Compared with past developments, the adaptive-sparse PDD
methods do not require truncation parameter(s) to be assigned
a priori or arbitrarily. In addition, two numerical techniques, one employing a nested sparse-grid dimension-reduction integration
and the other exploiting quasi MCS, were applied for the first time
to estimate the PDD expansion coefficients both accurately and efficiently.

The adaptive-sparse PDD methods were employed to calculate the second-moment
properties and tail probability distributions in three numerical problems,
where the output functions are either simple mathematical functions
or eigenvalues of dynamic systems, including natural frequencies of
a three-degree-of-freedom linear oscillator and an FGM plate. The
mathematical example reveals that the user-defined tolerances of an
adaptive-sparse PDD method are closely related to the relative error
in calculating the variance, thus providing an effective tool for
modulating the accuracy of the resultant approximation desired. Since
the adaptive-sparse PDD approximation retains only the component functions
with significant contributions, it is also able to achieve a desired
level of accuracy with considerably fewer coefficients than required by
existing truncated PDD approximations. The results of the linear oscillator
display a distinct advantage of the reduced ranking system over the
full ranking system, as the former requires significantly fewer expansion
coefficients to achieve results nearly identical to those of the latter. For
a required level of accuracy in calculating the tail probabilistic characteristics
of natural frequencies of an FGM plate, the new bivariate adaptive-sparse
PDD method is more economical than the existing bivariately truncated
PDD method by almost an order of magnitude. Finally, the new PDD method
was successfully applied to solve a stochastic dynamic instability
problem in a disk brake system, demonstrating the ability of the new
methods in handling industrial-scale problems.

\bibliographystyle{plain}
\bibliography{as_pdd_manuscript_arxiv}

\end{document}